\documentclass[11pt,letterpaper]{article}
\usepackage{latexsym,color,amssymb,amsthm,fullpage,enumitem,
amsmath,amsfonts,bm}

\usepackage{graphicx}
\usepackage{euscript}

  \usepackage[utf8]{inputenc}
\usepackage{graphicx}

\usepackage[left=1in, right=1in, top=1in, bottom=1in]{geometry}

\newcommand{\ignore}[1]{}

\newtheorem{theorem}{Theorem}[section]
\newtheorem{lemma}[theorem]{Lemma}
\newtheorem{conjecture}[theorem]{Conjecture}
\newtheorem{proposition}[theorem]{Proposition}
\newtheorem{claim}[theorem]{Claim}

\def\eps{{\epsilon}}

\def\conv{{{\mathtt{conv}}}}

\def\1{{(1)}}
\def\2{{(2)}}

\def\cwise{{\circlearrowright}}
\def\counterwise{{\circlearrowleft}}

\def\A{{\cal A}}

\def\F{{\cal F}}

\def\L{{\cal L}}
\def\HH{{\cal{H}}}
\def\C{{\EuScript{C}}}

\def\D{{\cal D}}
\def\B{{\cal B}}

\def\K{{{\mathcal {K}}}}

\def\Q{{\cal Q}}

\def\S{{\cal {S}}}

\def\M{{\cal M}}

\def\reals{{\mathbb R}}

\def\G{{\cal G}}

\def\T{{\cal T}}

\begin{document}

\begin{titlepage}

\title{{On Lines Crossing Pairwise Intersecting Convex Sets in Three Dimensions}\thanks{A preliminary version of this article appeared in the Proceedings of the 37th ACM-SIAM Symposium on Discrete Algorithms (SODA), January 11--14, 2026, Vancouver, Canada.}
}

\author{
Natan Rubin\thanks{Email: {\tt rubinnat.ac@gmail.com}. Ben Gurion University of the Negev, Beer-Sheba, Israel. Supported
by grants 2891/21 and 3010/25 from Israel Science Foundation.}
}

\maketitle

\begin{abstract}

The 1913 Helly's theorem states that any family $\K$ of $n\geq d+1$ convex sets in $\reals^d$ can be pierced by a single point if and only if any $d+1$ of $\K$'s elements can. In 2002 Alon, Kalai, Matou\v{s}ek and Meshulam ruled out the possibility of similar criteria for the existence of {\it lines} crossing multiple convex sets in dimension $d\geq 3$ -- for any $k\geq 3$, they described arbitrary large families $\K$ of convex sets in $\reals^3$ so that any $k$ elements of $\K$ can be crossed by a line yet no $k+4$ of them can.

Let $\K$ be a family of $n$ pairwise intersecting convex sets in $\reals^3$. We show that there exists a line crossing $\Theta(n)$ elements of $\K$. This resolves the most extensively studied variant of a problem by Mart\'inez, Rold\'an-Pensado and Rubin ({\it Discrete Comput. Geom. 2020}) which was highlighted by B\'ar\'any and Kalai ({\it Bull. Amer. Math. Soc. 2021}). Our result adds to the very few sufficient (and non-trivial) conditions that have been known for the existence of line transversals to large families of convex sets.

Our argument is based on a Ramsey-type result of independent interest for families of pairwise intersecting convex sets in $\reals^2$, and the structure of line arrangements in $\reals^3$.
\end{abstract}

\maketitle


\end{titlepage}

\section{Introduction} \label{sec:intro}

For any integer $n\geq 1$, we will use $[n]$ to denote the set $\{1,\ldots,n\}$. For any set $A$, and any integer $j$, we use $\binom{A}{j}$ to denote the collection of all the $j$-size subsets of $A$. 

Let $\K$ be a finite family of convex sets in $\reals^d$. We say that a collection $X$ of geometric objects (e.g., points, lines, or $k$-flats -- $k$-dimensional affine subspaces of $\reals^d$) is a \emph{transversal} to $\K$, and that $\K$ can be \emph{pierced} or \emph{crossed} by $X$, if each set of $\K$ is intersected by some member of $X$.

\subsection{Helly's theorem and its relatives}
A fundamental 1913 theorem in discrete and computational geometry and convex optimization, due to Helly \cite{Helly}, states that a finite family $\K$ of convex sets has a non-empty intersection (i.e., $\K$ can be pierced by a single point) if and only if each of its $(d+1)$-size subsets $\K'\in {\K\choose d+1}$ can be pierced by a point. In the past 60 years Geometric Transversal Theory has been preoccupied with the following questions (see e.g. \cite{AmentaLoSo,BaranyKalaiSurvey,Eckhoff,WengerHolmsen} for a comprehensive survey):

\begin{itemize}
 \item[1.] Given that a {\it significant fraction} of the $(d+1)$-tuples $\K'\in \binom{\K}{d+1}$ have a non-empty intersection, can $\K$, or at least some fixed fraction of its members, be pierced by constantly many points?
    \item[2.] Can an analogous ``Helly-type" statement be established for crossing families of convex sets with {\it $k$-flats}, for $1\leq k\leq d-1$?
\end{itemize}

\noindent{\bf Fractional Helly, $(p,q)$-theorem, and Colorful Helly theorem.} The most basic result of the first kind was established in 1979 by Katchalski and Liu.

\begin{theorem}[Fractional Helly's Theorem \cite{KatchalskiLiu}] \label{theorem:Fractional}
    For any $d\geq 1$ and $\alpha>0$ there is a number $\beta=\beta(\alpha,d)>0$ with the following property: For every finite family $\K$ of convex sets in $\reals^d$ so that at least $\alpha\binom{|\K|}{d+1}$ of the $(d+1)$-subsets $\K'\in \binom{\F}{d+1}$ have non-empty intersection, there is a point which pierces at least $\beta |\K|$ of the sets of $\K$.
\end{theorem}

Theorem \ref{theorem:Fractional} is one of the key ingredients in the 1992 proof of the so called Hadwiger-Debrunner $(p,q)$-Conjecture \cite{HadwigerDebrunner} by Alon and Kleitman \cite{AlonKleitman}. 

\begin{theorem}[The $(p,q)$-theorem \cite{AlonKleitman}] \label{theorem:PQ}
    For any $d\geq 1$ and $p\geq q\geq d+1$ there is a number $t=t(p,q,d)$ with the following property: Let $\K$ be a finite family of convex sets in $\reals^d$, so that any $p$-family $\K'\in {\K\choose p}$ encompasses a subset $\K''$ of $q$ elements with a non-empty intersection $\bigcap \K''$. Then $\K$ admits a transversal by $t$ points.
\end{theorem}

      We say that a ``$k$-colored" family $\K=\K_1\uplus\ldots\uplus \K_k$ of convex sets has the \emph{Colorful Helly property} if every ``colorful" choice $K_i\in \K_i$, for $1\leq i\leq k$, admits non-empty intersection $\bigcap_{i=1}^{k}K_i\neq \emptyset$.

\begin{theorem}[Lov\'asz's Colorful Helly's Theorem 1982 \cite{ImreColored}] \label{theorem:CHT}
    Let $\K=\K_1\uplus\ldots\uplus \K_{d+1}$ be a $(d+1)$-colored family of convex sets in $\reals^d$. Then Colorful Helly property implies that there is a color class $\K_i$ with non-empty intersection $\bigcap \K_i\neq \emptyset$.
\end{theorem}

\noindent{\bf Hyperplane transversals to families of convex sets.} The second question has been settled to the negative already for $k=1$. For instance, Santal\'o \cite{Santalo} and Danzer \cite{Danzer} observed that for any $n\geq 3$ there are families $\K$ of $n$ convex sets in $\reals^2$ so that any $n-1$ of the sets can be crossed by a single line transversal while no such transversal exists for $\K$. Nevertheless,
Alon and Kalai \cite{AlonKalai} were able to establish the following ``$(p,q)$-type" result.

\begin{theorem}[The $(p,q)$-theorem for hyperplanes \cite{AlonKalai}] \label{theorem:PQHyperplanes}
    For any $d\geq 1$ and $p\geq q\geq d+1$ there is a number $t=t(p,q,d)$ with the following property: Let $\K$ be a finite family of convex sets in $\reals^d$, so that any subset $\K'\in {\K\choose p}$ encompasses $q$ elements that can be crossed by a single hyperplane. Then the entire family $\K$ admits a transversal by $t$ hyperplanes.
    \end{theorem}

At the center of Alon-Kalai argument lies the following hyperplane analogue of Theorem \ref{theorem:Fractional}.

\begin{theorem}[A fractional Helly's theorem for hyperplanes \cite{AlonKalai}] \label{theorem:FractionalHyperplanes}
   For any $d\geq 1$ and $\alpha>0$ there is a number $\beta=\beta(\alpha,d)>0$ with the following property: For every finite family $\K$ of convex sets in $\reals^d$ so that at least $\alpha\binom{|\K|}{d+1}$ of the $(d+1)$-subsets $\K'\in \binom{\K}{d+1}$ can be crossed by a hyperplane, there is a hyperplane crossing at least $\beta |\K|$ of the sets of $\K$.
    \end{theorem}

\subsection{Lines crossing pairwise intersecting convex sets in $\reals^3$}

In 2002, Alon, Kalai, Matou\v{s}ek and Meshulam \cite{AKMM} showed that no analogue of Theorems \ref{theorem:PQHyperplanes} and \ref{theorem:FractionalHyperplanes} can exist for $k$-flats of {\it intermediate dimensionality} $k\in \{1,\ldots,d-2\}$ for $d\geq 3$: for every integers $d\geq 3,m\geq 3$ and $n_0\geq m+4$ there is a family of at least $n_0$ convex sets so that any $m$ of the sets can be crossed by a line but no $m+4$ of them can;
this phenomenon can be largely attributed to the complex topological structure of the space of transversal $k$-flats.

As ``standard" Helly-type theorems are out of question for line transversals in dimension $d\geq 3$, the following conjecture was proposed in 2018 by Mart\'inez, Rold\'an-Pensado and Rubin \cite{FurtherConsequences} (also see \cite[Conjecture 8.2]{BaranyKalaiSurvey})
while seeking stronger variants of Theorem \ref{theorem:CHT}.

\begin{conjecture}\label{Conj:2colored} 
There is an absolute constant $c>0$ with the following property.
	For any finite 2-colored family $\K=\K_1\uplus \K_2$ of convex sets in $\reals^3$, with the property that any two elements $K_1\in \K_1$ and $K_2\in \K_2$ have a non-empty intersection $K_1\cap K_2\neq \emptyset$, there is a color class $i\in \{1,2\}$ so that at least $c|\K_i|$ of its elements can be crossed by a single line.
\end{conjecture}

In subsequent years, much attention has focused on the following variant of Conjecture \ref{Conj:2colored}, which was highlighted in 2021 by B\'ar\'any and Kalai \cite[Conjecture 8.1]{BaranyKalaiSurvey}. 

\begin{conjecture}\label{Conj:Pairwise}
	There is an absolute constant $c>0$ with the following property.
	For any family $\K$ of $n$ convex sets in $\reals^3$, with the property that any two sets $K_1,K_2\in \K$ have a non-empty intersection $K_1\cap K_2\neq \emptyset$, there is a line crossing $cn$ of the sets $K\in \K$.
	\end{conjecture}

B\'ar\'any and Kalai \cite{BaranyKalaiSurvey} also supplied a construction of 5 pairwise intersecting convex sets with no common line transversal, which implies that the constant $c$ in Conjecture \ref{Conj:Pairwise} cannot be improved beyond 4/5 (even for large $n$).

\medskip

So far, the conjectures have been confirmed only for families of ``sufficiently round" convex sets (e.g., cylinders) \cite{cylinders} or, alternatively, flat sets that lie in parallel planes \cite{flats}. 

\medskip
To appreciate the difficulty of tackling Conjectures \ref{Conj:Pairwise} and \ref{Conj:2colored}, let us consider a relaxed fractional setting, in which only $\eps|K_1||\K_2|$ of the pairs $(K_1,K_2)\in \K_1\times \K_2$ intersect, for some $0<\eps\leq 1$. Specifically, if $\K_1$ is comprised of generic lines in $\reals^3$, and each element of $\K_2$ is crossed by $\eps|\K_1|$ lines of $\K_1$, then a transversal by $f(\eps)$ lines to $\K_2$ would yield a so called small-size weak $\eps$-net \cite{AlonSelections,AKMM} for line transversals in $\reals^3$ (also see Section \ref{Sec:Conclude}), which has been refuted in 2022 by Cheong and Goaoc and Holmsen \cite{CheongGoaocHolm}.

\subsection{\bf Our result} We establish Conjecture \ref{Conj:Pairwise}.

\begin{theorem}\label{Theorem:Main}
	There is an absolute constant $c>0$ with the following property.
	For any family $\K$ of $n$ convex sets in $\reals^3$, with the property that any two sets $K_1,K_2\in \K$ have a non-empty intersection $K_1\cap K_2\neq \emptyset$, there is a line crossing $cn$ of the sets $K\in \K$.
	\end{theorem}

 \begin{figure}[htb]
 \begin{center}
 	\includegraphics[scale=0.4]{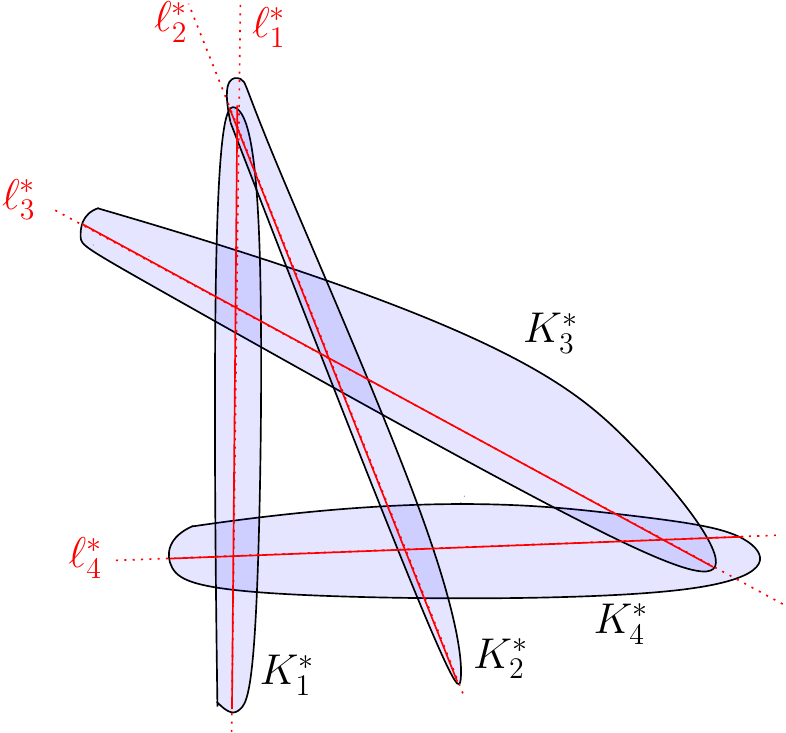}\hspace{0.5cm}\includegraphics[scale=0.4]{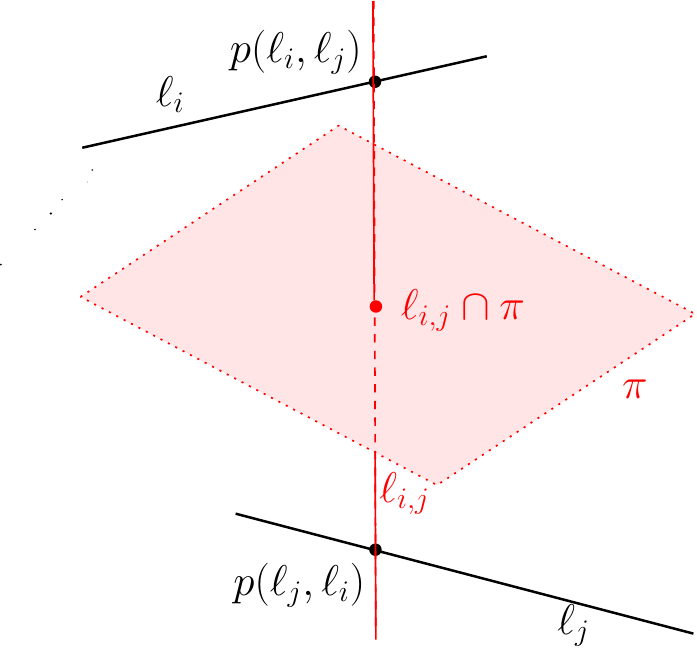}\hspace{0.5cm}	\includegraphics[scale=0.4]{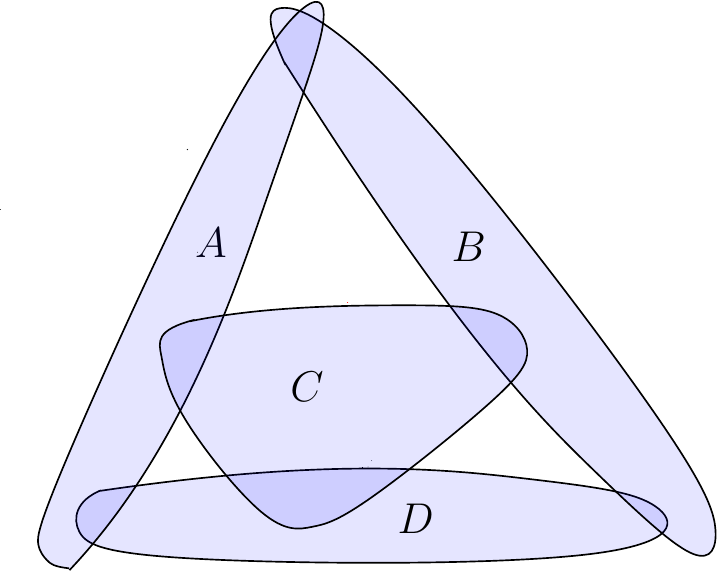}
 	\caption{\small Left: A projected strictly 2-intersecting family $\K^*=\{K_1^*,\ldots,K_4^*\}$ in $\reals^2$, and its realization by $4$ lines $\ell_1^*,\ldots,\ell^*_4$. Center: The lines $\ell_i$ and $\ell_j$ are separated by the plane $\pi$ -- the vertical line $\ell_{i,j}$ through $\ell_i$ and $\ell_j$ crosses $\pi$ in-between $p(\ell_i,\ell_j)$ and $p(\ell_j,\ell_i)$. Since $p(\ell_i,\ell_j)\in K_i$ and $p(\ell_j,\ell_i)\in K_j$, the intersection $\ell_{i,j}\cap \pi$ must lie in $K_i\cup K_j$. Right: The depicted strictly 2-intersecting family $\{A,B,C,D\}$ in $\reals^2$ cannot be realized, in the strong sense detailed above, by any 4 lines.}\label{Figure:realize}
 \end{center}
 \end{figure}
 
Here is a somewhat less formal sketch of our argument. Consider a vertical projection of $\K=\{K_1,\ldots,K_n\}$. If the projected family $\K^*=\{K^*_1,\ldots,K^*_n\}$ in $\reals^2$ encompasses $\Theta(n^3)$ triples $K^*_{i},K^*_{j},K^*_k$ that satisfy $K^*_{i}\cap K^*_{j}\cap K^*_k\neq \emptyset$, then Theorem \ref{theorem:Fractional} readily yields the desired {\it vertical} line in $\reals^3$ crossing $\Theta(n)$ elements of $\K$. Hence, it pays to focus on such families $\K$ whose vertical projections $\K^*$ are {\it strictly 2-intersecting} -- any two of their members intersect, yet no three of them have a common  intersection.

Let us assume for the sake of simplicity that the strictly 2-intersecting family $\K^*$ is ``line-like", and can be {\it realized} by a set $\{\ell^*_1,\ldots,\ell^*_n\}$ of $n$ lines in $\reals^2$ in the following strong sense: for any $1\leq i<j\leq n$ the intersection $\ell^*_i\cap \ell^*_j$ lies in $K_i^*\cap K_j^*$; see Figure \ref{Figure:realize} (left). Lifting the lines $\ell_i^*$ back to $\reals^3$ would then yield a line set $\L=\{\ell_1,\ldots,\ell_n\}$ with the following property:
for any $1\leq i<j\leq n$, the vertical line $\ell_{i,j}$ through $\ell_i$ and $\ell_j$ (which is raised over $\ell_i^*\cap \ell_j^*$) crosses both $K_i$ and $K_j$, within the respective intervals $K_i\cap \ell_i$ and $K_j\cap \ell_j$. 

The crucial observation is that there must exist a plane $\pi$ in $\reals^3$ that {\it separates} $\Theta(n^2)$ pairs $\{\ell_i,\ell_j\}\in {\L\choose 2}$, in the following sense: $\pi$ crosses $\ell_{i,j}$ in-between the points $p(\ell_i,\ell_j):=\ell_{i,j}\cap \ell_i$ and $p(\ell_j,\ell_i):=\ell_{i,j}\cap \ell_j$ (which belong, respectively, to $K_i$ and $K_j$). 
See Figure \ref{Figure:realize} (center). For each of these pairs $\ell_i,\ell_j$, at least one of the sets $K_i,K_j$ must cross $\pi$ in the relative vicinity of $v_{i,j}=\ell_{i,j}\cap \pi$.
To be precise, let $\L'$ denote the line family $\{\ell'_1,\ldots,\ell'_k\}$, where each element $\ell_k'$ is obtained by vertically lifting $\ell^*_k$ to $\pi$, and consider the subdivision of $\pi$ by $\L'_{i,j}:=\L'\setminus \{\ell_i,\ell_j\}$.
Then the cross-section of $K_i\cup K_j$ by $\pi$ can only lie in the cell that contains $v_{i,j}$ (or, else, some three members of $\K$ would be crossed by the same vertical line). 
Applying this observation to $\Theta(n^2)$ pairs $\ell_i,\ell_j$ of lines that are separated in the above sense by $\pi$. Using the convexity of each cross-section $K_i\cap \pi$, one can derive a subset $\L_{rich}$ of $\Theta(n)$ ``rich" lines $\ell_i\in \L$ whose respective sets $K_i$ are each crossed by $\Theta(n)$ lines $\ell'_j$ (and each time in the vicinity of the vertex $v_{i,j}=\ell'_i\cap \ell'_j$).
 A standard sampling argument then yields a transversal by $O(1)$ lines $\ell'_j$ to {\it all} such sets $K_i$ with $\ell_i\in \L_{rich}$.

Unfortunately, the outlined argument does not apply to general families of pairwise intersecting convex sets in $\reals^3$, as arrangements of strictly 2-intersecting families of convex sets in $\reals^2$ happen to be a far richer category (in topological terms) than line arrangements. Indeed, there exist families of just 4 sets that cannot be realized, in the strong sense detailed above, by any 4 lines; see Figure \ref{Figure:realize} (right).\footnote{See a recent study by \'{A}goston et al. \cite{Orientation} for a more comprehensive treatment of strictly 2-intersecting families of convex sets in $\reals^2$.} Instead, we establish the following Ramsey-type property: for any $n$, any sufficiently large strictly 2-intersecting family of convex sets in $\reals^2$ must contain a realizable (i.e., ``line-like") sub-family of size $n$. (Moreover, the $n$ lines in our realization form a monotone convex chain -- either an $n$-cap or an $n$-cup \cite[Section 3]{JirkaBook}.) The desired plane $\pi$, which contains a line crossing $\Theta(n)$ elements of $\K$, will be obtained by applying Theorem \ref{theorem:FractionalHyperplanes} of Alon and Kalai to a certain family $\K=\{K_{i,j}\}_{1\leq i<j\leq n}$, where each set $K_{i,j}$ encompasses all the vertical line segments $ab$ connecting $a\in K_i$ and $b\in K_j$.

\subsection{Paper organization}
 The rest of the paper is organized as follows. In Section \ref{Sec:Prelim} we introduce the basic notation, in which the proof of Theorem \ref{Theorem:Main}
  is cast, and formulate the key facts concerning strictly 2-intersecting families of convex sets in $\reals^2$, and arrangements of lines in $\reals^3$.

In Section \ref{Sec:Main} we establish Theorem \ref{Theorem:Main}. 

In Section \ref{Sec:Realize2d} we show that for any integer $n>0$, any sufficiently large family of strictly 2-intersecting family of convex sets in $\reals^2$ contains a $n$-size sub-family that can be realized by $n$ lines which form either a cap or cap.

In Section \ref{Sec:Separate}, we show that any family of $n$ lines which possess a certain monotone structure in $\reals^3$ must contain $\Theta(n^2)$ pairs that are separated by a single plane.\footnote{Though this is not essential for our proof of Theorem \ref{Theorem:Main}, the statement easily extends to {\it all} families of $n$ lines in $\reals^3$; see Section \ref{Sec:Conclude}.}

In Section \ref{Sec:Conclude}, we discuss the prospective extensions of our results to higher dimensions.

\section{Preliminaries}\label{Sec:Prelim}

 \medskip
 For any point $p\in \reals^d$, we use $p^*$ and $\hat{p}$ to denote, respectively, the vertical projection to $\reals^{d-1}$ and the vertical line $\{q\in \reals^d\mid q^*=p^*\}$ through $p$. Accordingly, for any set $A\subseteq\reals^d$, we denote $A^*:=\{p^*\mid p\in A\}$ and $\hat{A}:=\{p\mid p^*\in A^*\}=\bigcup_{p\in A} \hat{p}$.
 
Unless specified otherwise, for any non-vertical line $\ell$ in $\reals^2$ we use $\ell^+$ (resp., $\ell^-$) to denote the open half-plane above (resp., below) $\ell$, and the same convention applies to open halfspaces determined by planes in $\reals^3$.

 \medskip
 \noindent{\bf Ramsey's theorem.}   

\begin{theorem}[\cite{ErdSz,Ramsey}] \label{Theorem:Ramsey}For any integers $0<k\leq m$ and $c>0$ there exists an integer $R=R_k(m,c)\geq m$ with the following property:

 Any $c$-coloring $\chi:{[R]\choose k}\rightarrow [c]$ of the edges of the complete $k$-uniform hypergraph $\left([R],{[R]\choose k}\right)$ yields an $m$-subset $A\in {[R]\choose m}$ so that all the edges $e\in {A\choose k}$ are assigned the same color by $\chi$.	
\end{theorem}


\medskip
 \noindent{\bf Caps and cups of points and lines in $\reals^2$.} 
 We say that a sequence $p_1,\ldots,p_n$ of $n\geq 3$ points in general position in $\reals^2$ forms an {\it $n$-cap} (resp., {\it $n$-cup}) if the $x$-coordinates of the points $p_i$ form an increasing sequence and for all $1\leq i<j<k\leq n$ the point $p_j$ lies above (resp., below) the segment $p_ip_k$; see Figure \ref{Fig:cap-cup}.
 
 \begin{figure}[htb]
 \begin{center}
 	\includegraphics[scale=0.4]{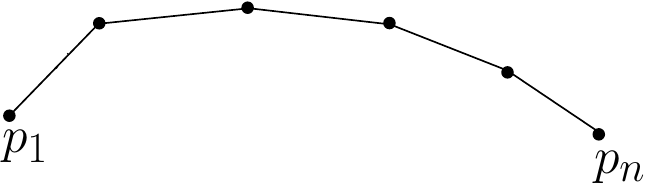}\hspace{2cm}\includegraphics[scale=0.4]{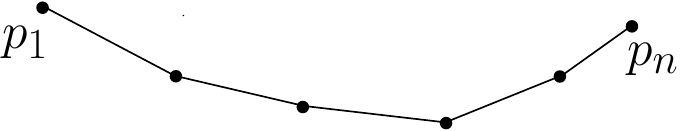}
 	\caption{\small Left: The sequence $p_1,\ldots,p_6$ of $n=6$ points forms a 6-cap. Right: A 6-cup of points.}\label{Fig:cap-cup}
 \end{center}
 \end{figure}
 
  \begin{figure}[htb]
 \begin{center}
 	\includegraphics[scale=0.4]{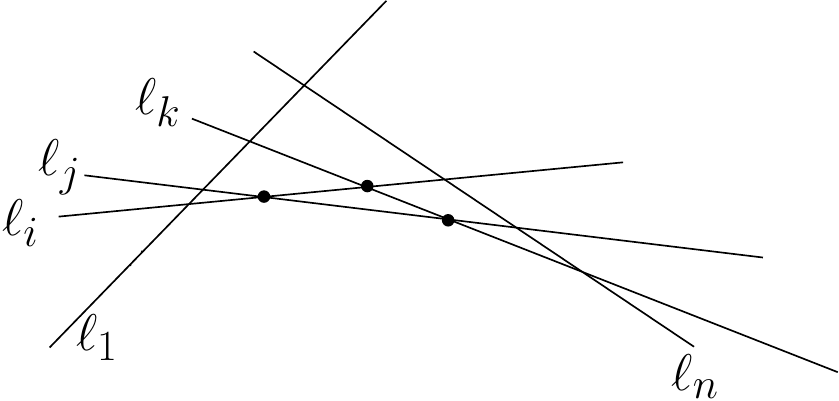}\hspace{1.5cm}\includegraphics[scale=0.4]{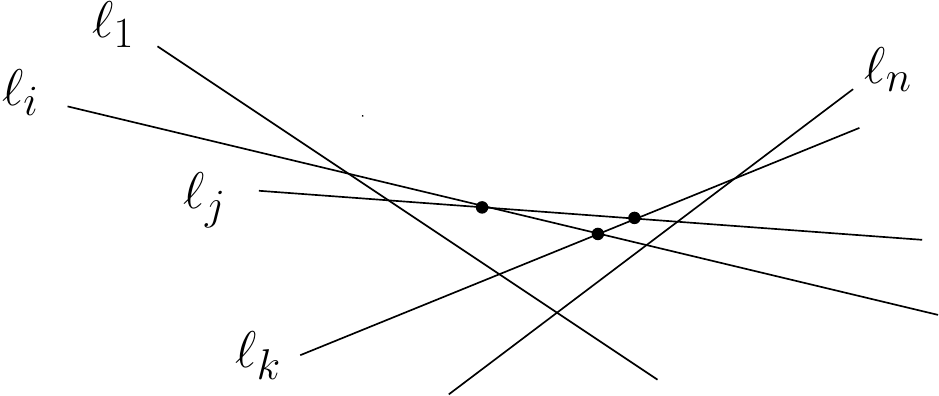}
 	\caption{\small An $n$-cap and an $n$-cup of lines (resp., left and right).}\label{Fig:cap-cup-lines}
 \end{center}
 \end{figure}
 
 We say that an ordered sequence of $n\geq 3$ non-vertical lines in $\reals^2$, so that no three of them pass through the same point, is an {\it $n$-cap} (resp., {\it $n$-cup}) if their respective slopes form a decreasing (resp., increasing) sequence and, furthermore, the intersection $\ell_i\cap \ell_k$ lies above (resp., below) $\ell_j$, for all $1\leq i<j<k\leq n$. See Figure \ref{Fig:cap-cup-lines}.


\subsection{Realizing sequences of planar convex sets with cap and cups} 
 We say that a finite family $\K$ of convex sets in $\reals^d$ is {\it strictly 2-intersecting} if every two of its sets have a non-empty common intersection, yet no three distinct sets of $\K$ do.
 
We say that a sequence $K_1,\ldots,K_n$ of $n\geq 3$ strictly $2$-intersecting convex sets in $\reals^2$ is {\it realized} by a sequence $\ell_1,\ldots,\ell_n$ of $n$ lines in $\reals^2$ if the segments $s_i=K_i\cap \ell_i$ and $s_j=K_j\cap \ell_j$ intersect whenever $1\leq i<j\leq n$. See Figure \ref{Fig:realize-sequence}.

 \begin{figure}[htb]
 \begin{center}
 	\includegraphics[scale=0.5]{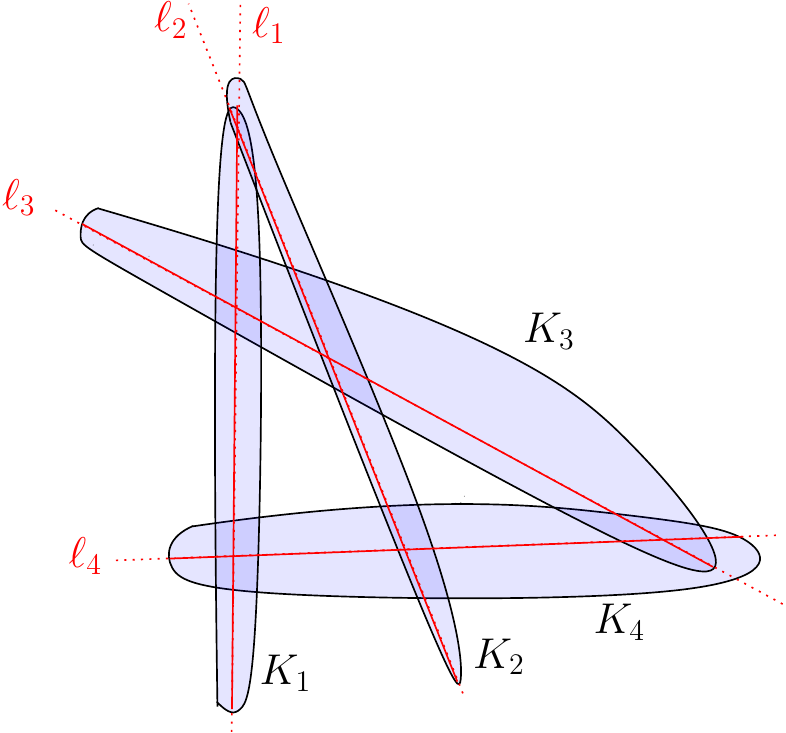}
 	\caption{\small The depicted sequence $K_1,\ldots,K_4$ is realized by 4 lines $\ell_1,\ldots,\ell_4$.}\label{Fig:realize-sequence}
 \end{center}
 \end{figure}



\medskip
In Section \ref{Sec:Realize2d} we establish the following strengthening of the Erd\H{o}s-Szekeres theorem \cite{ErdSz}.

\begin{theorem}\label{Theorem:Realize} For any positive integer $n$ there is a number $N_1=N_1(n)\geq n$ with the following property.

Any strictly 2-intersecting family $\K$ of $N_1\geq 3$ compact convex sets in $\reals^2$ 
contains a sequence $K_1,\ldots,K_n$ of $n$ distinct elements that can be realized by an $n$-line sequence $\ell_1,\ldots,\ell_n$ that is either an $n$-cap or an $n$-cup.
\end{theorem}

\subsection{Realizing a sequence of convex sets in $\reals^3$ by a monotone sequence of lines} 

\noindent{\bf Definition.} We say that a family $\K$ of convex sets in $\reals^3$ is {\it strictly 2-overlapping} if the projected family $\K^*:=\{K^*\mid K\in \K\}$ is strictly 2-intersecting.

A sequence of $n\geq 3$ pairwise intersecting convex sets in $\reals^3$ is {\it realized} by a sequence $\ell_1,\ldots,\ell_n$ of $n$ non-vertical and mutually non-parallel lines in $\reals^3$ if the projected segments $s^*_i=(K_i\cap \ell_i)^*$ and $s_j^*=(K_j\cap \ell_j)^*$ intersect whenever $1\leq i<j\leq n$. In other words, the vertical line $\ell_{i,j}$ through any pair $\ell_i$ and $\ell_j$ must pass through $K_i\cap \ell_i$ and $K_j\cap \ell_j$.

\begin{figure}[htb]
 \begin{center}
\includegraphics[scale=0.45]{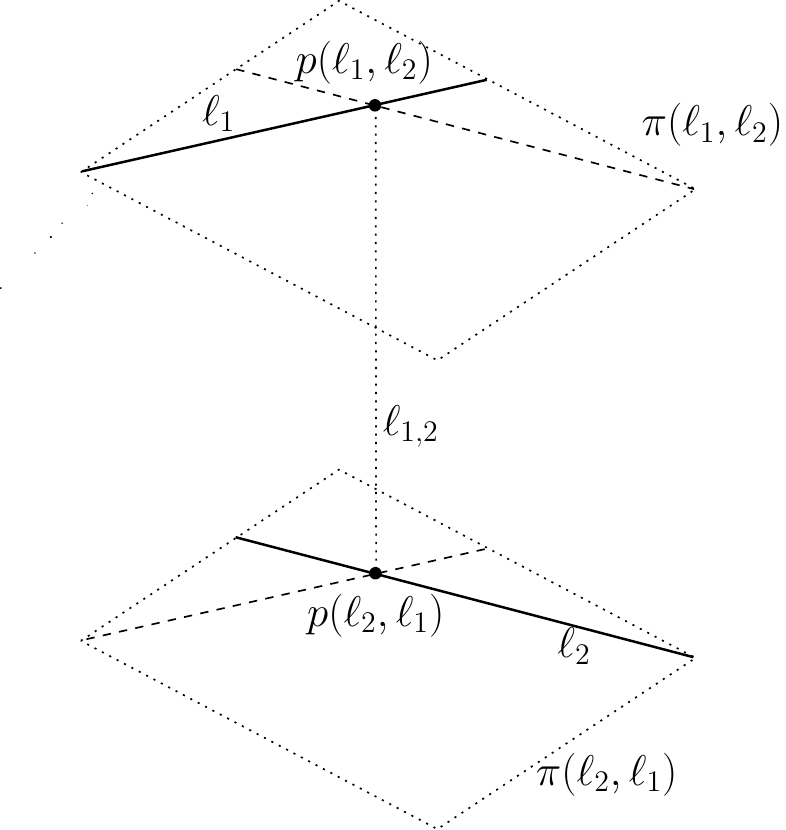}\hspace{2cm}\includegraphics[scale=0.45]{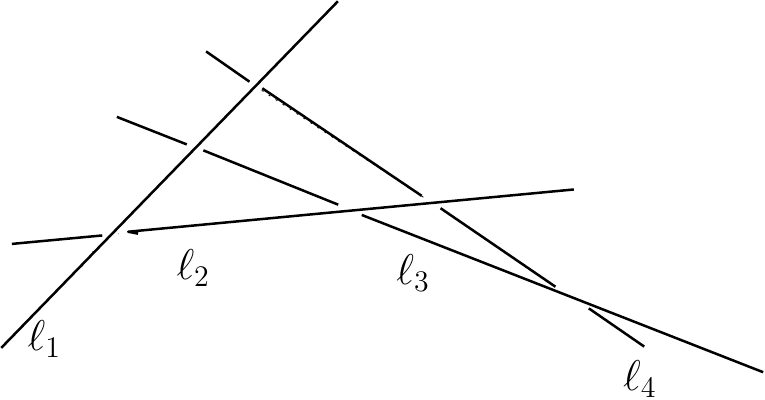}\\

 	\caption{\small Left: The lines $\ell_1$ and $\ell_2$ satisfy $\ell_1\succ \ell_2$. Depicted are the parallel planes $\pi(\ell_1,\ell_2)\supset \ell_1$ and $\pi(\ell_2,\ell_1)\supset \ell_2$, the unique vertical line $\ell_{1,2}$ through $\ell_1$ and $\ell_2$, and the points $p(\ell_1,\ell_2)=\ell_{1,2}\cap \ell_1$ and $p(\ell_2,\ell_1)\in \ell_{1,2}\cap \pi(\ell_2,\ell_1)$. Right: The sequence of 4 lines in $\reals^3$ is monotone, for we have that $\ell_1\succ \ell_2\succ\ell_3\succ \ell_4$, and the projections $\ell_1^*,\ldots,\ell^*_4$ form a 4-cap.}\label{Fig:Monotone}
 \end{center}
 \end{figure}

For any non-parallel pair of non-vertical lines $\ell_1,\ell_2$ in $\reals^3$, we use
$p(\ell_1,\ell_2)$ (resp., $p(\ell_2,\ell_1)$) to denote the intersection point between $\ell_1$ (resp., $\ell_2$) and the unique vertical line through $\ell_1$ and $\ell_2$, and we use $\pi(\ell_1,\ell_2)$ (resp., $\pi(\ell_2,\ell_1)$) to denote the $\ell_2$-parallel (resp., $\ell_1$-parallel) plane through $\ell_1$ (resp., $\ell_2$). See Figure \ref{Fig:Monotone} (left).


We say that $\ell_1$ lies {\it below} (resp., {\it above}) $\ell_2$, and denote $\ell_1\prec \ell_2$ (resp., $\ell_1\succ \ell_2$), if $p(\ell_1,\ell_2)$ lies below (resp., above) $p(\ell_2,\ell_1)$. Furthermore, we denote $\ell_1 \preceq \ell_2$ (resp., $\ell_1\succeq \ell_2$) if we have that either $\ell_1\prec \ell_2$ (resp., $\ell_1\succ \ell_2$) or $\ell_1\cap \ell_2\neq \emptyset$ (in which case $\pi(\ell_1,\ell_2)=\pi(\ell_2,\ell_1)$ holds).

\medskip
We say that a sequence $\ell_1,\ell_2,\ldots,\ell_n$ of $n\geq 3$ non-vertical lines in $\reals^3$ is {\it monotone} if the following two conditions are met (see Figure \ref{Fig:Monotone} (right)):

\begin{enumerate}
	\item The sequence $\ell^*_1,\ldots,\ell^*_n$ is either an $n$-cap or an $n$-cup. 
	\item We have that either $\ell_1\preceq \ell_2\preceq\ldots\preceq \ell_n$ or $\ell_1\succeq \ell_2\succ\ldots\succeq \ell_n$.
\end{enumerate}


\begin{theorem}\label{Lemma:Monotone} For any integer $n\geq 3$ there is a finite number $N_2=N_2(n)$ with the following property.

Any family $\K$ of $N_2$ strictly 2-overlapping convex sets in $\reals^3$ contains a sequence of $K_1,\ldots,K_n$ of $n$ sets that can be realized by a monotone sequence $\ell_1,\ldots,\ell_n$ of $n$ lines in $\reals^3$.
\end{theorem}

\begin{proof}
Let $m=R_2(n,2)$ and $N_2=N_1(m)$, where the quantity $N_1(\cdot)$ is defined in Theorem \ref{Theorem:Realize}. Let $\K$ be a strictly 2-overlapping family of $N_2$ sets in $\reals^3$.
Applying Theorem \ref{Theorem:Realize} to the projected strictly 2-intersecting family $\K^*=\{K^*\mid K\in \K\}$
 yields an $m$-sequence $F^*_1,\ldots,F^*_m$ which is realized by either an $m$-cap or an $m$-cup $\lambda_1,\ldots,\lambda_{m}$ of lines in $\reals^2$. For each $1\leq i\leq m$, we consider the interval $K^*_i\cap \lambda_i=a_ib_i$, raise vertical lines $\hat{a}_i$ and $\hat{b}_i$ in $\reals^3$ through the endpoints $a_i$ and $b_i$, respectively, and then ``lift" $\lambda_i$ back to a line $\mu_i$ in $\reals^3$ which passes through any pair of points $a'_i\in \hat{a}_i\cap F_i$ and $b'_i\in \hat{b}_i\cap F_i$. (We have that $a_i\neq b_i$ for all $1\leq i\leq m$, or else $\bigcap_{i=1}^m F_i^*$ would hold with $m\geq 3$.)
   
 Consider the 2-coloring of the edges in the complete graph $\left([m],{[m]\choose 2}\right)$, in which a pair of indices $\{i,j\}$, for $i<j$, is colored as {\it red} if $\mu_i\succeq \mu_j$, and it is colored as {\it blue} otherwise.
  Then, by Theorem \ref{Theorem:Ramsey}, and the choice of $m$, there must exist a ``monochromatic" subsequence ${i_1}<\ldots<i_{n}$ of $n$ indices, which yields a sequence $K_1=F_{i_1},\ldots,K_n=F_{i_n}$ of $n$ sets that is
realized in $\reals^3$ by the monotone subsequence $\ell_1=\mu_{i_1},\ldots,\ell_n=\mu_{i_n}$.
\end{proof}

\subsection{Separating families of lines and convex sets by planes in $\reals^3$}

\noindent{\bf Separating lines in $\reals^3$.} We say a skew pair of non-vertical lines $\ell_1$ and $\ell_2$ is {\it separated} by a non-vertical plane $H$ if the closed vertical segment $s(\ell_1,\ell_2):=p(\ell_1,\ell_2)p(\ell_2,\ell_1)$ is crossed by $H$.

\medskip
The following property is established in Section \ref{Sec:Separate}.

\begin{theorem}\label{Theorem:SeparateMany}
	For any monotone sequence $\ell_1,\ldots,\ell_n$ of $n$ lines in $\reals^3$, there must exist a plane that separates $\Theta(n^2)$ of the pairs $\ell_i,\ell_j$, for $1\leq i<j\leq n$.
\end{theorem}

\noindent{\bf Separating families of convex sets in $\reals^3$.} For any pair $A$ and $B$ of convex sets in $\reals^3$, let us denote 
$$
A\backsim B:=\conv \left(\left(A\cap \hat{B}\right)\cup \left(B\cap \hat{A}\right)\right).
$$
Notice that $A\backsim B$ can be described as the convex hull of the union of all the closed vertical segments $ab$ that connect a pair of points $a\in A$ and $b\in B$.

\medskip
\noindent{\bf Definition.} We say that convex sets $A,B$ in $\reals^3$ are {\it separated} by a plane $H$ if $H$ crosses the set $A\backsim B$. Notice that the two notions of separation coincide for skew pairs of non-parallel lines. 

\begin{lemma}\label{Lemma:VerticalHull} Let $A$ and $B$ be a pair of non-empty convex sets in $\reals^3$, and $H$ a plane.
Then $A$ and $B$ are separated by $H$ if and only if $H$ intersects\footnote{This includes containment if $H$ is vertical.}
 at least one closed vertical segment $ab$ with $a\in A$ and $b\in B$.
\end{lemma}

\begin{proof}
	Clearly, if the plane $H$ intersects a segment $ab$ with $a\in A$ and $b\in B$ then it clearly intersects $A\backsim B$ (which encompasses {\it all} such segments). For the converse direction, suppose for a contradiction that $H$ intersects $A\backsim B$ yet it intersects none of the segments $ab$ with $a\in A$ and $b\in B$. In particular, $H$ can intersect neither of the sets $A\cap \hat{B}$ and $B\cap \hat{A}$ (as each point $a\in A\cap \hat{B}$ is an endpoint of at least one vertical segment $ab$ with $b\in B$, and vice versa). Since $H$ intersects $A\backsim B$, the sets $A\cap \hat{B}$ and $B\cap \hat{A}$ must lie in distinct half-spaces of $\reals^3\setminus H$. Thus, $H$ is non-vertical and crosses {\it every} vertical segment $ab$ with $a\in A$ and $b\in B$.
	\end{proof}

  Notice that if a sequence $K_1,\ldots,K_n$ of convex sets is realized by a line sequence $\ell_1,\ldots,\ell_n$, then any plane $H$ that separates a pair of lines $\ell_i,\ell_j$, for $1\leq i\neq j\leq n$, must also separate the respective sets $K_i$ and $K_j$ (as $H$ crosses the closed vertical segment between $p(\ell_i,\ell_j)\in K_i\cap \ell_i$ and $p(\ell_j,\ell_i)\in K_j\cap \ell_j$). However, the vice versa does not necessarily hold.

\bigskip
\noindent{\bf Definition.} For any finite family $\K=\{K_1,\ldots,K_n\}$ of convex sets in $\reals^3$, we use $\K^\backsim$ to denote the family 
$\{K_i\backsim K_j\mid 1\leq i<j\leq q\}$.

\section{Proof of Theorem \ref{Theorem:Main}}	\label{Sec:Main}

It can be assumed with no loss of generality that the family $\K$ in question consists of {\it compact} convex sets. Otherwise, each set $K_i$ can be replaced by a convex polytope $K'_i=\conv\left(\{x_{i,j}\mid j\in [n]\setminus \{i\}\}\right)$, where $x_{i,j}$ denotes an arbitrary point in $K_i\cap K_j$. Indeed, the ``shrunken" family $\K'=\{K'_1,\ldots,K'_n\}$ is also pairwise intersecting, and any line transversal to a fixed fraction of the sets in $\K'$ must also cross the same proportion of $\K$.

\begin{theorem}\label{Theorem:pq}
For any integer $q\geq 3$ there is a constant $p\geq q$ with the following property.
Let $\K$ be a family of pairwise intersecting convex sets in $\reals^3$. Then at least one of the following conditions holds for every sub-family $\A
\in {\K\choose p}$:
\begin{enumerate}
	\item some three sets $A,B,C\in \A$ can be crossed by a vertical line (i.e., $A^*\cap B^*\cap C^*\neq \emptyset$), or
	\item there exists a sequence $K_1,\ldots,K_q$ of $q$ distinct elements of $\A$ that can be realized by a monotone sequence $\ell_1,\ldots,\ell_q$ of $q$ lines and, furthermore, at least $c{q\choose 2}$ of the pairs $K_i, K_j$, for $1\leq i<j\leq q$, can be separated by a single plane $\pi$; here $c$ is a constant that does not depend on $q$.
\end{enumerate}
\end{theorem}

\begin{proof}
Let $p=N_2(q)$, where $N_2(\cdot)$ is the quantity defined in Theorem \ref{Lemma:Monotone}. Let $\A$ be a subset of $p$ elements of $\K$.
Suppose that no 3 elements of $\A$ can be simultaneously crossed by a vertical line, so the family $\A$ is strictly $2$-overlapping. By Theorem \ref{Lemma:Monotone} (and the choice of $p$ and $q$, which satisfy the relation in the theorem), $\A$ must contain a sequence $K_1,\ldots,K_q$ of $q$ distinct elements that can be realized by a monotone sequence $\ell_1,\ldots,\ell_q$ of $q$ lines. By Theorem \ref{Theorem:SeparateMany}, there must exist a plane $\pi$ separating at least $cq^2$ pairs $\ell_i,\ell_j$ (for $1\leq i<j\leq q$) and, therefore, also $cq^2$ of the respective pairs $K_i,K_j$ of convex sets.
\end{proof}

At the center of our argument lies the following property whose proof is postponed to the end of this section.

\begin{lemma}\label{Lemma:3sets}
For any $c_0>0$, the followings property holds with any sufficiently large integer $q>0$. Let $\Q=\{K_1,\ldots,K_q\}$ be a strictly $2$-overlapping family of $q$ pairwise intersecting convex sets in $\reals^3$. Suppose that the sequence $K_1,\ldots,K_q$ is realized by a line sequence $\ell_1,\ldots,\ell_q$. Let $\pi$ be a plane separating at least $c_0{q\choose 2}$ pairs $\{K_i,K_j\}\in {\Q\choose 2}$.\footnote{Notice that $\pi$ is not required to separate the respective pairs $\ell_i$ and $\ell_j$ of lines.} Then there exists a line $\ell\subset \pi$ crossing at least $3$ elements of $\Q$.
\end{lemma}

\noindent{\bf Definition.} Let $d\geq 1$ be an integer. Following the notation of Alon and Kalai \cite{AlonKalai}, we say that a family $\F$ of $d$ convex sets in $\reals^d$ is {\it good} if we have that either $d=1$, or $\F$ cannot be crossed by a $(d-2)$-dimensional flat.\footnote{In the notation of Cappell {\it et al.} \cite{Cappell}, such a family is called {\it separated}.} 
Furthermore, a family of $d+1$ convex sets in $\reals^d$ is called {\it good} if all of its $d$-size subsets are good.

We say that a hyperplane $\pi$ is {\it tangent} to a convex set $K$ in $\reals^d$ if $K\cap \pi\neq \emptyset$ and $K$ lies in one of the closed half-spaces determined by $\pi$.
If the hyperplane $\pi$ is {\it oriented}, the two open halfspaces are pre-labelled as $\pi^-$ and $\pi^+$.
The following property was established by Cappell {\it et al.} \cite[Theorem 3]{Cappell}.

\begin{lemma}\label{Lemma:Tangents}
Any good family $\F$ of $d\geq 1$ compact convex sets in $\reals^d$ supports exactly $2^{d+1}$ distinct oriented hyperplanes that are tangent to each member of $\F$.
\end{lemma}

In fact, Cappell {\it et al.} \cite[Theorem 3]{Cappell} proved the following stronger statement: for any partition $\F=\F^-\uplus \F^+$, the exist two oriented common tangent hyperplanes $\pi_1$ and $\pi_2$ to $\F$ so that $\bigcup \F^-$ (resp., $\bigcup \F^+$) is contained in the closures of $\pi_1^-$ and $\pi_2^-$ (resp., the closures of $\pi_1^+$ and $\pi_2^+$).
Using Lemma \ref{Lemma:Tangents}, Alon and Kalai \cite[Lemma 4.3]{AlonKalai} derived the following property of good families of $d+1$ convex sets.

\begin{lemma}
Any good family $\F$ of $d+1$ convex sets in $\reals^d$ that can be crossed by a hyperplane, contains a $d$-size subset $\F'$ so that at least one of its $2^{d+1}$ oriented common tangent hyperplanes crosses the remaining element of $\F\setminus \F'$.
\end{lemma}

\noindent{\bf Definition.} 
For any $0\leq k\leq d$, let $\varphi_{d,k}:\reals^d\rightarrow \reals^k$ denote the standard projection of $\reals^d$ to the copy of $\reals^k$ that is spanned by the first $k$ coordinate axes. (That is, we have that $x^*=\varphi_{d,d-1}(x)$ for all $x\in \reals^d$.)

We say that a hyperplane $\pi$ is {\it determined} by a family $\A$ of compact convex sets if there is a subset $\B\subset \K$ of $1\leq i\leq d$ elements so that 
\begin{enumerate}
	\item $\B'=\{\varphi_{d,i}(K)\mid K\in \B\}$ is a good family in $\reals^i$, and
	\item $\pi=\varphi_{d,i}^{-1}(\pi')$, where $\pi'$ is a common tangent hyperplane to $\B'$ within $\reals^i$.
\end{enumerate}

Notice that any family $\A$ of $d$ convex sets yields a collection $\HH(\A)$ of at most $\sum_{i=1}^d{d\choose i}\cdot 2^i\leq 3^d$ hyperplanes that are determined by $\A$ via its various subsets $\B$. The following stronger variant of Theorem \ref{theorem:FractionalHyperplanes} was implicitly established by Alon and Kalai \cite[Proof of Proposition 4.1]{AlonKalai}. For the sake of completeness, its proof is spelled out in Appendix \ref{App:FracDefined}.

\begin{lemma}\label{Lemma:DefinedTangent}
For any $0<\alpha\leq 1$, and any integer $d\geq 1$, there is a constant $\beta'_d(\alpha)$ with the following property.
Let $\K$ be a family of $n\geq d+1$ convex sets in $\reals^d$ that encompasses $\alpha{n\choose d+1}$ subsets $\A\in {\K\choose d+1}$, each crossed by a hyperplane. Then there is a hyperplane crossing $\beta'_d(\alpha)|\K|$ elements of $\K$ that is determined by a subset $\B\in {\K\choose d}$.
\end{lemma}

To establish Theorem \ref{Theorem:Main}, we fix a suitably large integer constant $q\geq 3$, which will be determined in the sequel, and select $p\geq q$ in accordance with Theorem \ref{Theorem:pq}. 
It can be assumed with no loss of generality that $|\K|\geq p$ or, else, 
the theorem holds with any line crossing an element of $\K$.
If there exist at least $\frac{1}{2}{n\choose p}$ $p$-size subsets $\A\in {\K\choose p}$ so that each of them encompasses at least one triple $\{A,B,C\}$ with $A^*\cap B^*\cap C^*\neq \emptyset$, then the entire family $\K$ must encompass at least ${n\choose p}/{n\choose p-3}=\Omega(n^3)$ triples $\{A,B,C\}$ of this kind; hence, applying the Fractional Helly Theorem (Theorem \ref{theorem:Fractional}) to the family $\K^*=\{K^*\mid K^*\in \K\}$ yields a vertical line crossing $\Omega(n)$ elements of $\K$. 

Therefore, and since $\K$ is  pairwise intersecting, it can be assumed in what follows that at least $\frac{1}{2}{n\choose p}$ of the $p$-size subsets $\A\in {K\choose p}$ are strictly $2$-overlapping.
For each $p$-subset $\A\in {\K\choose p}$ of this kind, Theorem \ref{Theorem:pq} yields a subset $\Q=\{K_1,\ldots,K_q\}$ of $q$ distinct elements of $\A$ which can be realized by a set $\L=\{\ell_1,\ldots,\ell_q\}$ of $q$ lines, and a plane separating at least $c{q\choose 2}$ among the sets $K_{i,j}=K_i\backsim K_j$ of $\Q^\backsim$. Repeating this argument for the at least $\frac{1}{2}{n\choose p}$ available choices of $\A\in {\K\choose p}$ yields a collection $\Gamma(\K,q)\subseteq {\K\choose q}$ of ${n\choose p}/(2{n\choose p-q})=\Omega(n^q)$ families $\Q$ with the aforementioned properties, where the implicit constant of proportionality depends on $p$ and $q$. 

For each $q$-family $\Q\in \Gamma(\K,q)$, applying Lemma \ref{Lemma:DefinedTangent} to the induced family $\Q^\backsim$ yields a collection $\D_{\Q}=\{K_{i_1,i_2},K_{i_3,i_4},K_{i_5,i_6}\}$ of $3$ convex sets which determines a transversal plane $\pi_{\Q}$ to at least $\beta_3(c){q\choose 2}$ elements of $\Q^\backsim$. 
Hence, by the pigeonhole principle, there must exist a sub-family $\Q_0\in {\K\choose 6}$, a subset $\D_0\in {\Q_0^\backsim\choose 3}$, and a plane $\pi_0\in \HH(\D_0)$, so that $\D_\Q=\D_0$ and $\pi_{\Q}=\pi_0$ holds for $\Omega(n^{q-6})$ different choices $\Q\in \Gamma(\K,q)$ of the form $\Q=\Q_0\cup \S$, with $\S\in {\K\setminus \Q_0\choose q-6}$. Provided a suitably large choice of the constant $q$, this readily yields a collection $\Lambda_0(\K,q-6)$ of $\Theta(n^{q-6})$ subsets $\S\in {\K\setminus \Q_0\choose q-6}$ so that each of them meets the criteria of Lemma \ref{Lemma:3sets} with the plane $\pi_0$ (albeit with some smaller constant $c_0<\beta_3(c)$, as $\pi_0$ crosses at least $\beta_3(c){q\choose 2}-6q$ elements of $\S^\backsim$). Hence, given a sufficiently large choice of $q$, each of these families $\S\in \Lambda_0(\K,q-6)$ must contain a subset $\T\subseteq \S$ of $3$ elements that can be crossed by a single line within $\pi_0$. Since any such triple $\T\in {\K\choose 3}$ can arise from only $O(n^{q-9})$ distinct choices of $\S\in \Lambda_0(\K,q-6)$, the family $\K$ must contain $\Theta(n^3)$ triples $\T$ of this kind. Now Theorem \ref{Theorem:Main} follows by applying Theorem \ref{theorem:FractionalHyperplanes} to the 2-dimensional family $\{K\cap \pi_0\mid K\in \K\}$. $\Box$

\bigskip
\noindent{\bf Proof of Lemma \ref{Lemma:3sets}.} To each pair $\{K_i,K_j\}$, with $1\leq i<j\leq q$, we assign a point $y_{i,j}\in K_i\cap K_j$.
Let $\Pi$ denote the subset of all the pairs $\{K_i,K_j\}\in {\Q\choose 2}$ that are separated by $\pi$. According to Lemma \ref{Lemma:VerticalHull}, for each pair $\{K_i,K_j\}\in \Pi$ one can fix a closed vertical segment $a_{i,j}b_{i,j}$, with $a_{i,j}\in K_i$ and $b_{i,j}\in K_j$, which intersects $\pi$ at some point $x_{i,j}$.\footnote{Notice that it is nowhere assumed that $a_{i,j}$ and $b_{i,j}$ lie on, respectively, $\ell_i\cap K_i$ and $\ell_j\cap K_j$, or that $\ell_i$ and $\ell_j$ are separated by $\pi$.} Notice that the segments $y_{i,j}a_{i,j}\subset K_i$ and $y_{i,j}b_{i,j}\subset K_j$, whose vertical projections lie entirely in $K_i^*\cap K_j^*$,  comprise a path $\gamma_{i,j}$ crossing $\pi$ at a single point $w_{i,j}\in K_i\cup K_j$ (so that $w_{i,j}^*\in K_i^*\cap K_j^*$).

\medskip
At the center of our argument lies the 2-dimensional line family $\L'=\{\ell'_1,\ldots,\ell'_q\}$ within $\pi$, which is comprised of the respective vertical projections $\ell'_i$ of the lines $\ell_i\in \L$ onto the plane $\pi$.\footnote{Notice that, since $\K$ is realized by $\ell_1,\ldots,\ell_q$, and it is strictly 2-overlapping, no three lines of either set $\L',\L^*=\{\ell_1^*,\ldots,\ell_q^*\}$ can intersect at the same point, and no two of them can be parallel.} For each pair $\{\ell'_i,\ell'_j\}\in {\L'\choose 2}$, with $1\leq i<j\leq q$, we use $v_{i,j}$ to denote the intersection point $\ell'_i\cap \ell'_j$, whose vertical projection $\ell_i^*\cap \ell_j^*$ also lies in $K_i^*\cap K_j^*$.

With some abuse the notation, for any $1\leq i<j\leq q$ we refer to the point $w_{i,j}$ as either $w_{i,j}$ or $w_{j,i}$, and the same symmetrical notation applies to the points $v_{i,j}$.

For each $1\leq i\leq q$, let us denote
$$
W:=\{w_{i,j}\mid \{K_i,K_j\}\in \Pi\},
$$

$$
W_i:=W\cap K_i,
$$


and

$$
V_i=\{v_{i,j}\mid \{K_i,K_j\}\in \Pi,w_{i,j}\in W_i\}.
$$


Within each line $\ell'_i\in \L'$, 
 we consider the {\it open} interval $I_i=v_{i,a_i}v_{i,b_i}$ between the two extremal points of $V_i$, and note that there exist a total of at least 
 $$
 \sum_{i=1}^q \left(|V_i|-2\right)=\sum_{i=1}^q\left(|W_i|-2\right)=|\Pi|-2q\geq c_0{q\choose 2}-2q
 $$ 
 \noindent pairs $(I_i,\ell'_h)$ in which the interval $I_i$ is crossed by the line $\ell'_h\in \L'$, each time at the point $v_{i,h}$ (where $i\in [q]$ and $h\not\in \{a_i,b_i\}$). Hence,  
 choosing $q>10/c_0$ yields a line $\ell'_h\in \L'$ crossing at least $(c_0{q\choose 2}-2q)/q\geq 3$ distinct intervals $I_i$, each at the respective vertex $v_{i,h}\in V_i$, so that $i\in [q]\setminus \{h\}$ and $h\not\in \{a_i,b_i\}$. It suffices to show that for each of these intervals $I_i$, the line $\ell'_h\subset \pi$ must also cross the respective segment $w_{i,a_i}w_{i,b_i}\subset K_i\cap \pi$.

 To this end, let us fix an index $i\in [q]\setminus \{h\}$ so that $v_{i,h}\in I_i\cap V_i$ and $v_{i,h}\not\in \{v_{i,a_i},v_{i,b_i}\}$, and suppose for a contradiction that $\ell'_h$ misses the segment $w_{i,a_i}w_{i,b_i}\subset K_i\cap \pi$. Then at least one of the segments $v_{i,a_i}w_{i,a_i}$ or $v_{i,b_i}w_{i,b_i}$, let it be the former one, must cross the line $\ell'_h$; see Figure \ref{Fig:InsidePi}.
  We will show that $K^*_i\cap K_{a_i}^*\cap K_{h}^*\neq \emptyset$, which is contrary to the prior assumption that the family $\Q$ is strictly 2-overlapping.

       \begin{figure}
 \begin{center}
\includegraphics[scale=0.45]{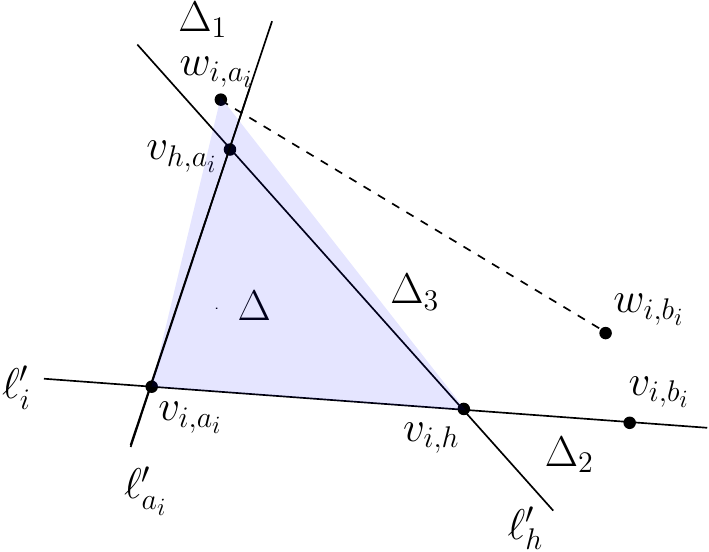}\hspace{1.5cm}\includegraphics[scale=0.45]{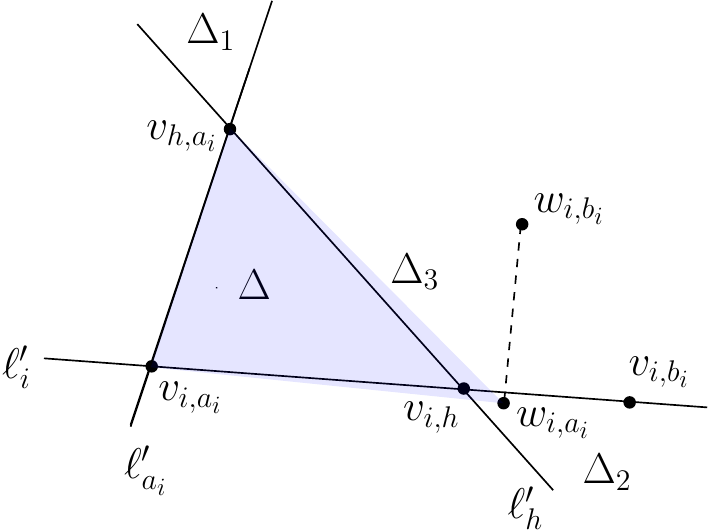}\\

\vspace{1cm}
\includegraphics[scale=0.45]{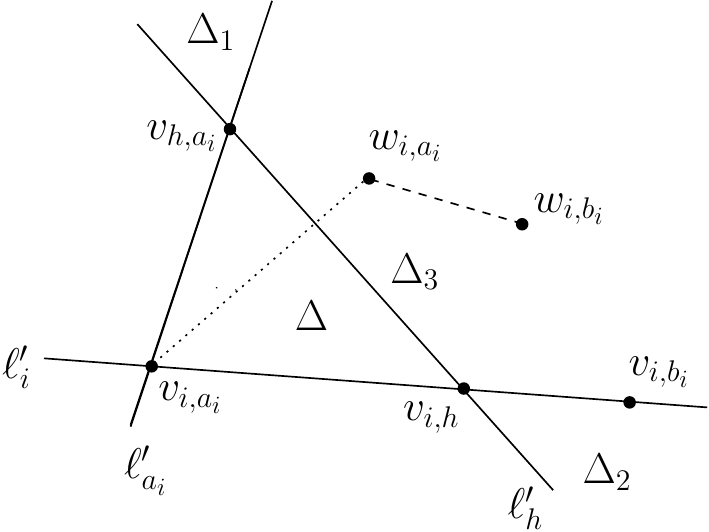}
 	\caption{\small Proof of Lemma \ref{Lemma:3sets} -- view within the plane $\pi$. Each line $\ell'_i$ is a vertical projection of $\ell_i$ to $\pi$.  If the segment $v_{i,a_i}w_{i,a_i}$ crosses $\ell'_h$, then we argue that $K^*_i\cap K^*_{a_i}\cap K^*_h\neq \emptyset$, which is contrary to the choice of $\Q$ as a strictly 2-overlapping family in $\reals^3$. Depicted are the three possible scenarios depending on the location of $w_{i,a_i}$ within $(\ell'_h)^+$.}\label{Fig:InsidePi}
 \end{center}
 \end{figure}

 Indeed, consider the arrangement of the three lines $\ell'_i$, $\ell'_{a_i}$ and $\ell'_h$, which we orient so that $\Delta=(\ell_i')^-\cap(\ell'_{a_i})^-\cap (\ell'_{h})^-$ comprises the only bounded cell -- a triangle. Since $v_{i,a_i}$ lies in $(\ell'_h)^-$, the point $w_{i,a_i}$ must lie in the other halfplane $\ell_h^+$. As such, it must lie in closure of one of the cells $\Delta_1:=(\ell'_{a_i})^+\cap (\ell'_h)^+\cap (\ell'_i)^-$, $\Delta_2:=(\ell'_{a_i})^-\cap (\ell'_h)^+\cap (\ell'_i)^+$, or $\Delta_3:=(\ell'_{a_i})^-\cap (\ell'_h)^+\cap (\ell'_i)^-$; see Figure \ref{Fig:InsidePi}.
 However, if $w_{i,a_i}$ lies in the closure of $\Delta_1$, then the point $v_{a_i,h}$ lies in the triangle $\tau=\conv(\Delta\cup \{w_{i,a_i}\})=\conv\left(v_{i,a_i}v_{i,h}\cup\{w_{i,a_i}\}\right)$ whose vertical projection $\tau^*$ lies in $K_i^*$, for both $w_{i,a_i}$ and $v_{i,a_i}v_{i,h}\subset a_ib_i\subset \ell'_i$ project to $K^*_i$. Hence, we have that $v^*_{a_i,h}\in K^*_i\cap K^*_{a_i}\cap K^*_h$.
    A symmetric argument applies if $w_{i,a_i}$ lies in the closure of $\Delta_2$ -- we argue that the point $v_{i,h}$ falls in the triangle $\conv(\Delta\cup \{w_{i,a_i}\})$, whose projection lies in $K^*_{a_i}$, so that $v_{i,h}^*\in K^*_{a_i}\cap K^*_i\cap K^*_h$. Lastly,  $w_{i,a_i}$ lies in $\Delta_2$, then the intersection $v_{i,a_i}w_{i,a_i}\cap v_{i,h}v_{a_i,h}$ again projects to a point of $K^*_i\cap K_{a_i}^*\cap K_{h}^*$. $\Box$


\section{Proof of Theorem \ref{Theorem:Realize}}\label{Sec:Realize2d}

As was explained in the beginning of Section \ref{Sec:Main}, it can be assumed with no loss of generality that the strictly 2-intersecting family $\K$ in question consists of compact convex sets (or, else, each set can be replaced by a convex polygon).

\medskip
Let us first establish the basic properties of the arrangements of three or four strictly 2-intersecting convex sets in $\reals^2$.

  \begin{figure}[htbp]
 \begin{center}
\includegraphics[scale=0.55]{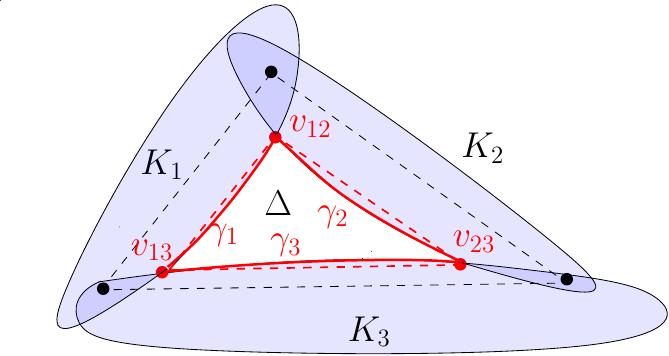}
 	\caption{\small Lemma \ref{Lemma:OrientSets}. The complement $\reals^2\setminus (K_1\cup K_2\cup K_3)$ of the strictly 2-intersecting family $\K=\{K_1,K_2,K_3\}$ contains exactly one bounded region, which is delimited by the arcs $\gamma_i\subset\partial K_i$.}\label{Fig:3Hole}
 \end{center}
 \end{figure}

\begin{lemma}[\cite{Orientation,Petruska}]\label{Lemma:OrientSets}
 Let $K_1,K_2,K_3$ be a strictly 2-intersecting ordered triple of compact convex sets in $\reals^2$. Then the following statements hold (see Figure \ref{Fig:3Hole}).
 
 \begin{enumerate}
 	\item[i.] The complement $\reals^2\setminus (K_1\cup K_2\cup K_3)$ can be represented as the union $\Delta\cup \Delta_\infty$ which consists of an unbounded region $\Delta_\infty$ and a bounded and simply connected hole $\Delta=\Delta(K_1,K_2,K_3)$ whose boundary is comprised of 3 arcs $\gamma_1\subseteq \partial K_1$, $\gamma_2\subseteq \partial K_2$, $\gamma_3\subseteq\partial K_3$. 
 	\item[ii.] For every choice of $x_{i,j}\in K_i\cap K_j$ with $1\leq i<j\leq 3$, the oriented triple $x_{1,2},x_{2,3},x_{1,3}$ has the same orientation (i.e., clockwise or counterclockwise), and $\conv(x_{1,2},x_{2,3},x_{1,3})$ contains $\Delta$.  
 	\item[iii.] The closure of $\conv(\Delta)$ is a triangle $\triangle v_{12}v_{23}v_{13}$, where $v_{ij}$ denotes the common endpoint of $\gamma_i$ and $\gamma_j$. Moreover, for each $1\leq i\leq 3$ the interior of the segment $v_{ij}v_{ik}$ (with $\{j,k\}=[3]\setminus \{i\}$) is contained in $K_i\setminus (K_j\cup K_k)$.
 
  \end{enumerate}
 \end{lemma}

 Accordingly, we say that a strictly 2-intersecting triple $(K_1,K_2,K_3)$ of compact convex sets in $\reals^2$ has {\it clockwise} (resp., {\it counterclockwise}) orientation, and denote this by $\cwise(K_1,K_2,K_3)$
 (resp., $\counterwise(K_1,K_2,K_3)$), if any three points $x_{1,2}\in K_1\cap K_2,x_{2,3}\in K_2\cap K_3$, $x_{1,3}\in K_1\cap K_3$ appear in clockwise (resp., counterclockwise) order along the boundary of the triangle $\conv(x_{1,2},x_{2,3},x_{3,1})$.

 \begin{lemma}\label{Lemma:4sets}
Let $K_1,\ldots,K_4$ be a strictly 2-intersecting sequence of 4 compact convex sets in $\reals^2$, so that no two of them are mutually tangent, and suppose that all the ordered triples $K_{i_1},K_{i_2},K_{i_3}$, for $1\leq i_1<i_2<i_3\leq 4$, have the same orientation. 
Let $[4]=\{a,b\}\uplus \{c,d\}$ be a partition so that $\Delta_c=\Delta (K_a,K_b,K_c)$ and $\Delta_d=\Delta(K_a,K_b,K_d)$ are adjacent to the same vertex $v$ of $\partial (K_a\cup K_b)$, and overlap in its arbitrary small neighborhood. For $i\in \{c,d\}$, let $\gamma_i$ denote the arc of $\partial \Delta_i$ that is contained in $\partial K_i$, with endpoints $v_{ai}\in \partial(K_{a}\cap K_i)$ and $v_{bi}\in \partial(K_b\cap K_i)$. Furthermore, suppose that at least one of the arcs $\gamma_c$ or $\gamma_d$ contains a point of $K_d\cap K_c$. 
Then we have that $v_{ac}v_{bc}\cap v_{ad}v_{bd}\neq \emptyset$.
\end{lemma}

 \begin{figure}[htb]
 \begin{center}
\includegraphics[scale=0.4]{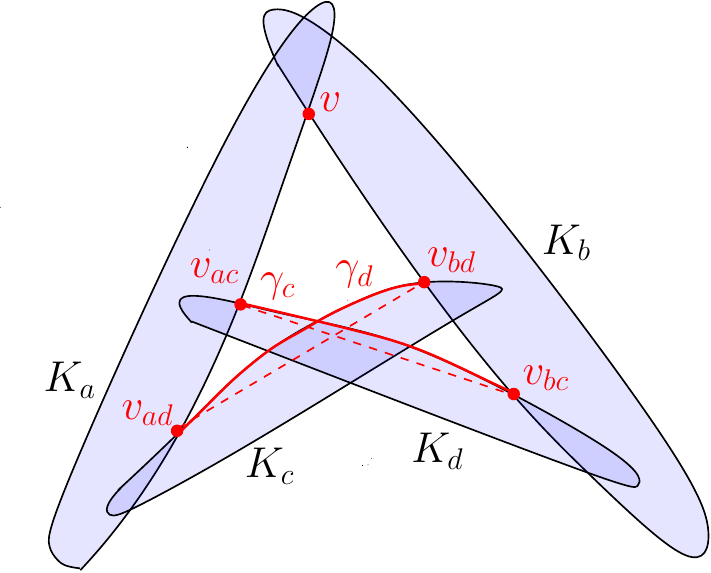}\hspace{2cm}\includegraphics[scale=0.4]{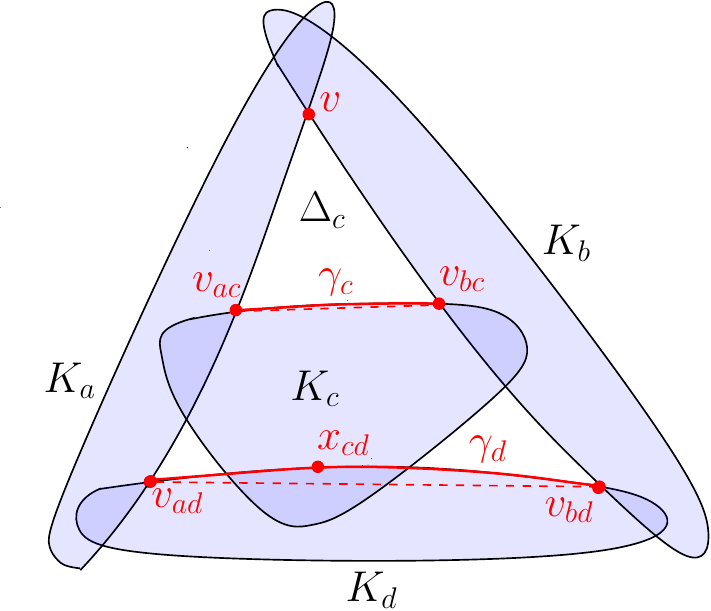}
 	\caption{\small Lemma \ref{Lemma:4sets}. Left: Both $\Delta_c=\Delta(K_a,K_b,K_c)$ and $\Delta_d=\Delta(K_a,K_b,K_d)$ are adjacent to the same vertex $v$ of $\partial (K_a\cup K_b)$, and overlap in its viccinity. Furthermore, at least one of the arcs $\gamma_c,\gamma_d$ contains a point of $K_c\cap K_d$. Then we have that $v_{ac}v_{bc}\cap v_{ad}v_{bd}\neq \emptyset$. Right: Proof of the lemma. If $v_{ac}v_{bc}\cap v_{ad}v_{bd}= \emptyset$ then some two sub-sequences among $K_1,\ldots,K_4$ must have opposite orientations.}\label{Fig:4sets}
 \end{center}
 \end{figure}

\begin{proof}

Assume for a contradiction that the claim does not hold, so that $v_{ac}v_{bc}\cap v_{ad}v_{bd}=\emptyset$. Refer to Figure \ref{Fig:4sets}.
Since both open regions $\Delta_c$ and $\Delta_d$ are adjacent to the same vertex $v\in \partial (K_a\cap K_b)$, in whose vicinity they overlap, it follows that the ordered triples $v,v_{bc},v_{ac}$ and $v,v_{bd},v_{ad}$ attain the same orientation, let it be clockwise. 
Since, $v_{ac}v_{bc}\cap v_{ad}v_{bd}=\emptyset$ (and the interiors of these segments cannot cross $K_a\cup K_b$ in view of Lemma \ref{Lemma:3sets} (iii)), it can be assumed with no loss of generality that $v,v_{ac}$ and $v_{ad}$ (resp., $v,v_{bc}$ and $v_{bd}$) appear in this clockwise (resp., counterclockwise) along $\partial K_a$ (resp., $\partial K_b$). Thus, $\gamma_d$ must contain a point $x_{cd}\in K_c\cap K_d$.
Up to circular and mirror symmetry, there are only three fundamental scenarios to consider:
\begin{enumerate}
	\item If $a=1,b=2,c=3$ and $d=4$, then Lemma \ref{Lemma:3sets} implies that the triples $(K_1,K_3,K_4)$ and $(K_2,K_3,K_4)$ have opposite orientations, which as can be seen by examining the oriented point triples $(v_{13},x_{34},v_{14})$ and $(v_{23},x_{34},v_{24})$.
	\item If $a=1,b=3,c=2$ and $d=4$, then the triples $(K_1,K_2,K_3)$ and $(K_1,K_2,K_4)$ again have opposite orientations (as is demonstrated by
	the ordered point triples $(v,v_{12},v_{23})$ and $(v_{12},x_{24},v_{14})$). 
	\item If $a=1,b=3,c=4$ and $d=2$ then, similar to the previous case, the triples $(K_1,K_3,K_4)$ and $(K_1,K_2,K_4)$ have opposite orientations.
	\end{enumerate}
\end{proof}

We say that a strictly 2-intersecting ordered triple $A,B,C$ of convex sets in $\reals^2$ is a {\it 3-cap} (resp., {\it 3-cup}) if every $3$ points $a\in A\cap B$, $b\in A\cap C$, and $c\in B\cap C$ form a 3-cap (resp., 3-cup) $a,b,c$. \footnote{In particular, this means that no two distinct intersections among $A\cap B,A\cap C,B\cap C$ can be crossed by a vertical line.}

Let  $n_1=c_1\cdot \max\{n,R_3(n+1,2)\}$ and $N_1=R_3(n_1,8)$, where $R_k(m,c)$ denotes the Ramsey number in Theorem \ref{Theorem:Ramsey} and $c_1$ is a suitably large constant to be fine-tuned in the sequel.
To establish Theorem \ref{Theorem:Realize}, let us fix a strictly 2-intersecting family $\K$ of $N_1\geq 3$ compact convex sets in $\reals^2$.
Consider an arbitrary ordering $K_1,\ldots,K_{N_1}$ of the sets of $\K$. 

Let $\T=\{K_i,K_j,K_k\}\in {\K\choose 3}$, for $i<j<k$. Since no three of the intersections $K_i\cap K_j$, $K_i\cap K_j$ and $K_{j}\cap K_k$ can be crossed by a single line (or, else, the 1-dimensional Helly variant would yield $\bigcap \T\neq \emptyset$), any such triple $\T$ must fall into exactly one of the following categories:
\begin{enumerate}
	\item[(i)] Some permutation of $\T$ yields a 3-cap.
	\item[(ii)] Some permutation of $\T$ is a 3-cup.
	\item[(iii)] Two of the sets $K_i\cap K_j$, $K_i\cap K_j$ and $K_{j}\cap K_k$, say $K_i\cap K_k$ and $K_i\cap K_j$, are crossed by a vertical line which lies to the left of $K_j\cap K_k$. 
	\item[(iv)]  Two of the sets $K_i\cap K_j$, $K_i\cap K_j$ and $K_{j}\cap K_k$, say $K_i\cap K_k$ and $K_i\cap K_j$, are crossed by a vertical line which lies to the right of $K_j\cap K_k$. 
\end{enumerate}
 Combined with the two possible orientations $\cwise(K_i,K_j,K_k)$ and $\counterwise(K_i,K_j,K_k)$, which derive from Lemma \ref{Lemma:3sets}, this yields an 8-coloring of all triples in ${\K\choose 3}$ (see Figure \ref{Fig:ColorTriples}). Since $N_1\geq R_3(n_1,8)$, Theorem \ref{Theorem:Ramsey} yields a subset $\K_1\subset \K$ of $n_1$ elements so that all triples $\T\in {\K_1\choose 3}$ have the same color. Up to symmetry, two possible scenarios arise:
 \begin{enumerate}
 	\item All of the triples $\T=\{K_{i},K_j,K_k\}\in {\K_1\choose 3}$, for $i<j<k$, are of type (i) and, in addition, attain the same orientation $\counterwise(K_i,K_j,K_k)$.
 	\item All of the triples $\T=\{K_{i},K_j,K_k\}\in {\K_1\choose 3}$, for $i<j<k$, are of type (iii) and, in addition, attain the same orientation $\cwise(K_i,K_j,K_k)$.
 \end{enumerate} 

  \begin{figure}[htb]
 \begin{center}
 	 	\includegraphics[scale=0.45]{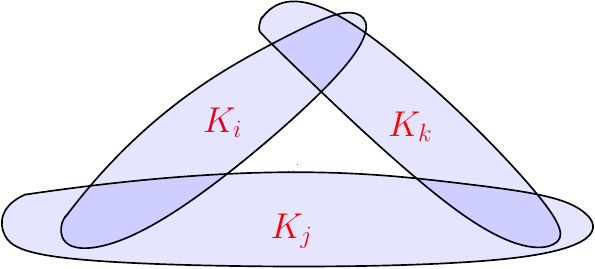}\hspace{1cm}\includegraphics[scale=0.45]{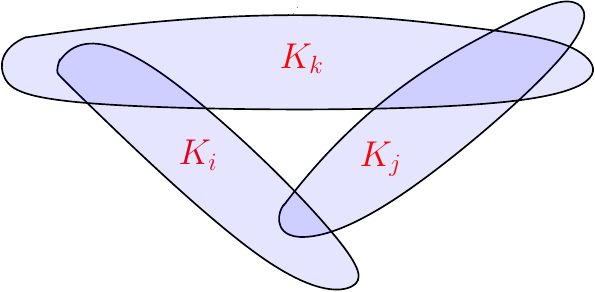}\hspace{1cm}\includegraphics[scale=0.45]{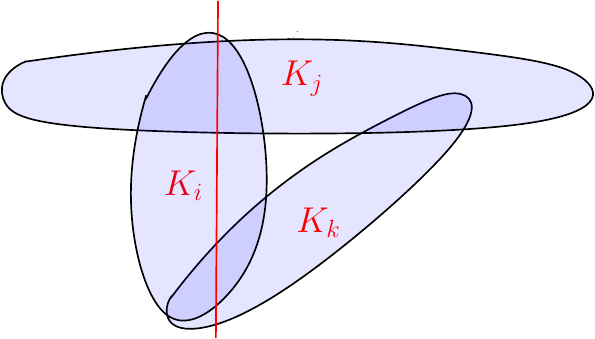}
 	 
 	\caption{\small Proof of Theorem \ref{Theorem:Realize} -- three ``colors" of triples $\{K_i,K_j,K_k\}$, with $i<j<k$, are illustrated. Left: The triple $\{K_i,K_j,K_k\}$ is of type (i), and its orientation is $\counterwise(K_i,K_j,K_k)$. Center: The triple $\{K_i,K_j,K_k\}$ is of type (ii), and its orientation is $\counterwise(K_i,K_j,K_k)$. Right: The triple $\{K_i,K_j,K_k\}$ is of type (iii), and its orientation is $\cwise(K_i,K_j,K_k)$.}\label{Fig:ColorTriples}
 \end{center}
 \end{figure}

\noindent {\bf Case 1.} By the pigeonhole principle, there must exist elements $A=K_{i},B=K_{k}$ in $\K_1$ that are involved in some $m\geq 3{n_1\choose 3}/{n_1\choose 2}=\Theta(n_1)$ 3-caps of the form $(A,K_j,B)$. A suitable choice of the constant $c_1>0$ will guarantee that $m\geq n$.

We fix $A$ and $B$, and consider the family $\F$ which is comprised of all the $m$ such choices of $K_j$. 

For each element $K_j\in \F$, we use $\gamma_j$ to denote the sub-arc $(\partial K_j)\cap (\partial \Delta(A,K_j,B))$. We then use $a_j$ and $b_j$ to denote the endpoints of $\gamma_j$ that lie on, respectively, $A$ and $B$, and $\ell_j$ to denote the line through $a_j$ and $b_j$. Lemma \ref{Lemma:3sets} (ii) implies that all of the regions $\Delta(A,K_j,B)$ are adjacent to the same vertex $v$ of $\partial(A\cup B)$, and overlap in the vicinity of $v$. (This is because both $\gamma_j$ and $a_jb_j$ are crossed by the downward ray through $v$ that leaves $A\cap B$.) 

We re-arrange the sets $K_j\in \F$ in the decreasing order of the slopes of their respective lines $\ell_j$, which yields the sequence $F_1,\ldots,F_{m}$. With nominal abuse of notation, we also use $\gamma_j$ to denote the arc of $\partial \Delta(A,F_j,B)$ that is ``supported" by $F_j$, and whose endpoints $a_j$ and $b_j$ lie on, respectively,  $\partial A$ and $\partial B$, and accordingly keep using $\ell_j$ to denote the line through $a_j$ and $b_j$. See Figure \ref{Fig:capcap} (left).

\begin{claim}\label{Claim:Cross}
	Let $1\leq j\neq j'\leq m$. Then the segments $a_jb_j\subset \ell_j\cap F_j$ and $a_{j'}b_{j'}\subset \ell_{j'}\cap F_{j'}$ cross.
\end{claim}

 \begin{figure}[htb]
 \begin{center}
 	 	\includegraphics[scale=0.4]{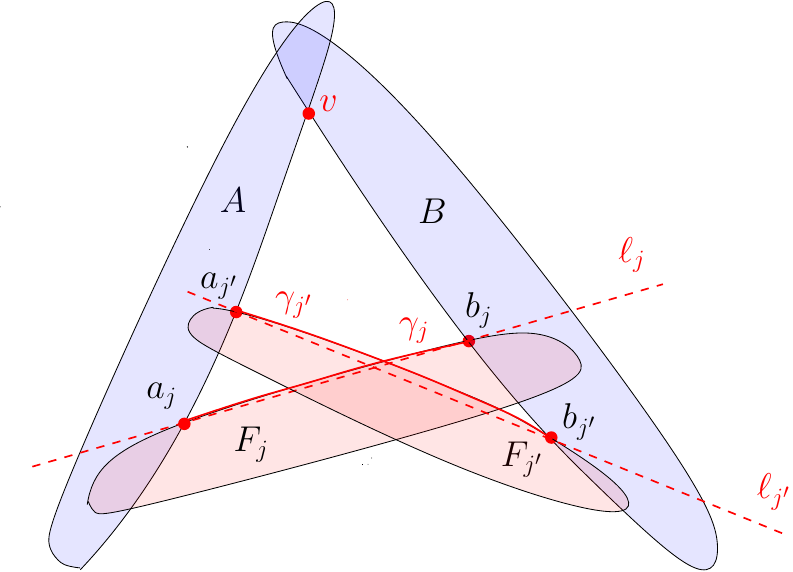}\hspace{1cm}\includegraphics[scale=0.4]{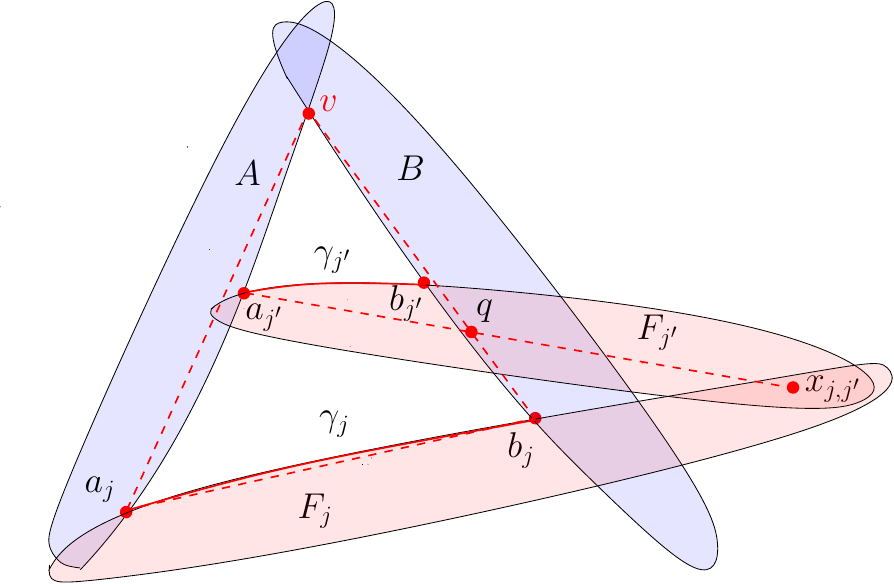}
 	 
 	\caption{\small Proof of Theorem \ref{Theorem:Realize} -- case 1. Left: Claim \ref{Claim:Cross} -- for any $j<j'$, the segments $a_jb_j\subset \ell_j\cap F_j$ and $a_{j'}b_{j'}\subset \ell_{j'}\cap F_{j'}$ cross. Right: Proof of Claim \ref{Claim:Cross}. If $a_jb_j\cap a_{j'}b_{j'}=\emptyset$, and there is a point $x_{j,j'}\in F_j\cap F_{j'}$ so that $a_{j'}x_{j,j'}$ crosses $vb_j$. Hence, $B,F_{j'}$ and $F_j$ form a forbidden cup.}\label{Fig:capcap}
 \end{center}
 \end{figure}

\begin{proof}[Proof of Claim \ref{Claim:Cross}]
Assume for a contradiction that $a_jb_j\cap a_{j'}b_{j'}=\emptyset$ so, by Lemma \ref{Lemma:4sets}, neither of the arcs $\gamma_j,\gamma_{j'}$ can contain a point of $F_j\cap F_{j'}$.
Fix an intersection point $x_{j,j'}\in F_j\cap F_{j'}$.   It can be assumed, with no loss of generality, that both $\gamma_{j'}$ and $a_{j'}b_{j'}$ lie within the cap triangle $\triangle a_jvb_j$, and above $\gamma_j$, whereas $x_{j,j'}$ lies outside $\triangle a_jvb_j$. Furthermore, neither of the segments $a_{j'}x_{j,j'}$ and $b_{j'}x_{j,j'}$ can cross the bottom edge $a_jb_j$ (or, else, this segment would also cross the arc $\gamma_j$ at a point of $F_j\cap F_{j'}$). It, thereby, follows that that either $a_{j'}x_{j,j'}$ crosses $vb_j$ or $b_{j'}x_{j,j'}$ crosses $va_j$. Let us assume with no loss of generality that the former scenario occurs, and notice that (i) the intersection point $q=vb_j\cap b_{j'}x_{j,j'}$ belongs to $F_{j'}\cap F_k$, and (ii) the point $b_j$ must lie below the segment $qx_{j,j'}$ (by the convexity of $F_j$). Hence, the three points $q\in F_{j'}\cap B,b_j\in F_j\cap B$, and $x_{j,j'}\in F_j\cap F_{j'}$ form a 3-cup. However, then the ordered triple $(B,F_{j'},F_j)$ too must form a 3-cup; see Figure \ref{Fig:capcap} (right). A fully symmetric argument yields the 3-cup $(F_{j},F_{j'},A)$ if $b_{j'}x_{j,j'}$ crosses $va_j$. Either cup is contrary to the ``monochromatic" choice of $\F\subset \K_1$.
\end{proof}

Now let $1\leq j<h<l\leq m$. Then the three points $\ell_j\cap \ell_h\in F_j\cap F_h$, $\ell_{j}\cap \ell_l\in F_j\cap F_l$ and $\ell_h\cap \ell_l\in F_h\cap F_l$ must form a 3-cap rather than a 3-cup (or, else, no permutation of $\{F_j,F_h,F_l\}$ would yield a 3-cap). Thus, the sequence $F_1,\ldots,F_m$ is realized by the line sequence $\ell_1,\ldots,\ell_m$, which is an $m$-cap.

\medskip
\noindent {\bf Case 2.} Since all ${n_1\choose 3}$ of the triples $\{K_i,K_j,K_k\}\subset \K_1$ are of type (iii), there must exist a set $A=K_i$ that is involved in $t=\Omega(n_1^2)$ ordered triples $(A,K_j,K_k)$ with the following property: the intersections $K_i\cap K_j$ and $K_i\cap K_k$ are crossed by a vertical line $\ell$, whereas $K_j\cap K_k$ lies to the right of $\ell$. By vertically projecting all the intersections $K_j\cap A$, with $K_j\in \K_1\setminus \{A\}$,  and invoking the 1-dimensional fractional Helly's theorem for the resulting intervals (so that any intersecting pair of intervals correspond to pairs of sets $A\cap K_j$ and $A\cap K_k$ that can be simultaneously crossed by a vertical line), we obtain a vertical line $\ell$ crossing some $m=\Omega(n_1)$ intersections  $A\cap K_j$, with $\K_1\setminus \{A\}$. Once again, a suitable choice of $c_1$ yields $m\geq R_3(n+1,2)$. We use $\F$ to denote the family comprised of the $m$ respective convex sets $K_j$. For each $K_j\in \F$, we fix an arbitrary point in $K_j\cap A\cap \ell$, and relabel the elements of $\F$ as $F_1,\ldots,F_m$ in the ascending order of their respective points $a_1\in F_1\cap A\cap \ell,\ldots,a_m\in F_m\cap A\cap \ell$.\footnote{Since $\K_1$ is strictly 2-intersecting, no two of these intervals $F_j\cap A\cap \ell$ can overlap.} 

Notice that for all $1\leq j<k\leq m$ the intersection $F_j\cap F_k$ lies entirely to the right of $\ell$ (as $\{A,F_j,F_k\}$ is of type (iii) and, in addition $\ell$ cannot cross $F_j\cap F_k$ if the triple is strictly 2-intersecting). Let us now construct a secondary coloring of the triples $j<h<k$ in $[m]$. To this end, we apply Lemma \ref{Lemma:3sets} to the triple $a_ja_k,F_j$ and $F_k$, whose ``hole region" $\Delta_{j,k}=\Delta(a_ja_k,F_j,F_k)$ is adjacent to some boundary vertex $v_{jk}$ of $F_j\cup F_k$. 

Note that $F_h$ intersects exactly one of the arcs $\gamma_j\subset \partial F_j$ and  $\gamma_k\subset \partial F_k$ that are adjacent to $v_{j,k}$ within $\partial \Delta_{j,k}$; see Figure \ref{Fig:Vertical}. Indeed, since both $F_h\cap F_j$ and $F_h\cap F_k$ lie to the right of $\ell$, and the point $a_h$ lies in $F_h\cap \ell$, the set $F_h$ must intersect $\gamma_j\cup \gamma_k$ (or, else, it would never ``exit" $\Delta_{j,k}$ to meet $F_j$ and $F_k$ to the right of $\ell$). However, $F_h$ cannot intersect {\it both} $\gamma_j$ and $\gamma_k$. Indeed, otherwise the region $\Delta(F_j,F_h,F_k)$ would too be adjacent to $v_{j,k}$, and overlap $\Delta_{j,k}$ in the vicinity of $v_{j,k}$. According to Lemma \ref{Lemma:4sets}, the endpoints of the arc $(\partial K_h)\cap (\partial \Delta(F_j,F_h,F_k))$ would then span a segment crossing $a_ja_k$. As a result, one of the regions $F_h\cap F_j$ or $F_h\cap F_k$ would encompass a point to the left of (or on) $\ell$.

\begin{figure}[htb]
 \begin{center}
 	 	\includegraphics[scale=0.4]{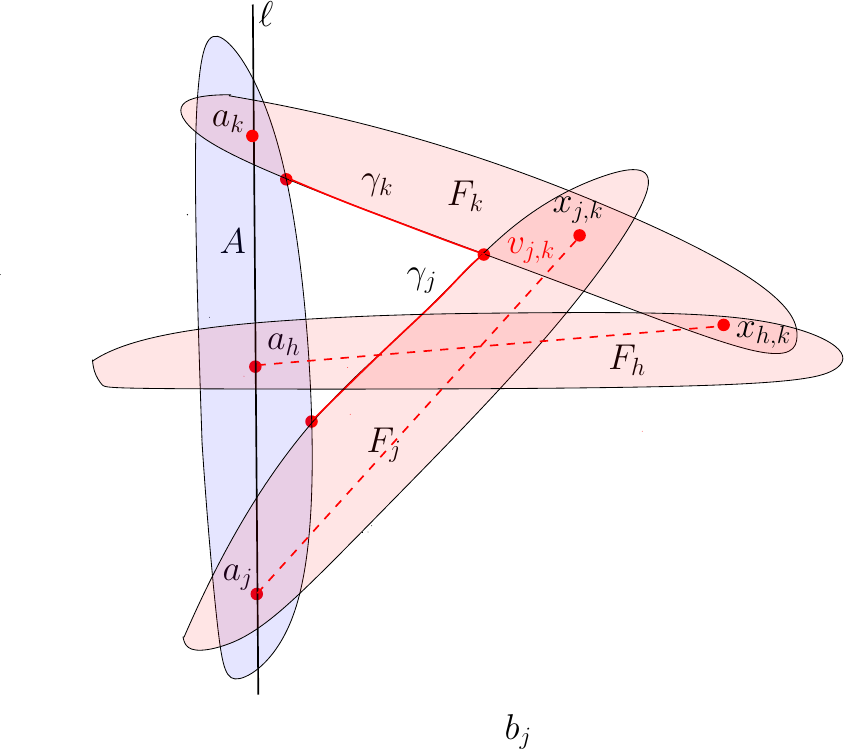}\hspace{1cm}\includegraphics[scale=0.4]{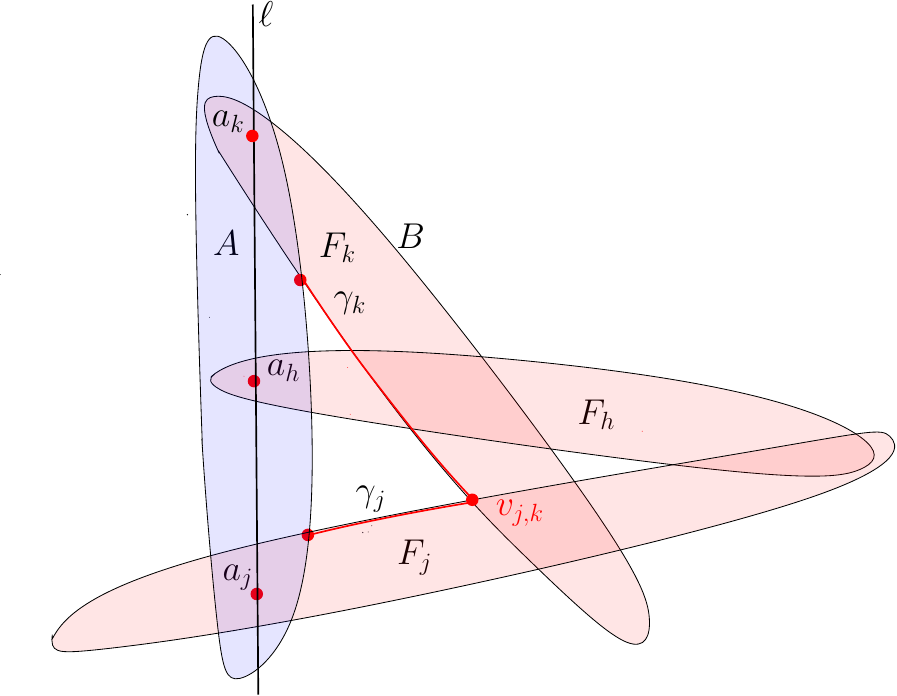}
 	 
 	\caption{\small Proof of Theorem \ref{Theorem:Realize} -- case 2: We color the triple $\{F_j,F_h,F_k\}$ as {\it red} (left) or {\it blue} (right) depending on which of the arcs $\gamma_j,\gamma_k$ is crossed by $F_h$.}
 	\label{Fig:Vertical}
 \end{center}
 \end{figure}

We thus color the triple $\{F_j,F_h,F_k\}$ as {\it red} (resp., {\it blue}) if $F_h$ intersects $\gamma_j$ (resp., $\gamma_k$). 
Since $m\geq R_3(n+1,2)$, another application of Theorem \ref{Theorem:Ramsey} yields an $(n+1)$-size subset $\G\subset \F$ so that all of the triples $\{F_j,F_h,F_k\}\subset \G$ are of the same color, say red. We relabel the elements of $\G$ as $G_1,\ldots,G_{n+1}$ in the ascending vertical order of their representatives $a_i\in G_i\cap \ell$. 
\begin{claim}\label{Claim:CrissCross}
	Let $1\leq j<h<k\leq n+1$. Then for any choice $x_{h,k}\in G_j\cap G_h$ and $x_{j,k}\in G_j\cap G_k$ we have that $a_hx_{h,k}\cap a_jx_{j,k}\neq \emptyset$.
\end{claim}

\begin{proof}[Proof of Claim \ref{Claim:CrissCross}.]
Notice that the segment $a_hx_{h,k}$ intersects exactly one of the edges $a_kx_{j,k}\subset G_k,a_jx_{j,k}\subset G_j$ of $\triangle a_kx_{j,k}a_j$. 
However, since the triple $\{F_j,F_h,F_k\}$ is red, the directed segment $\overline{a_hx_{j,k}}$ leaves the region $\Delta(a_ja_k,F_j,F_k)$ (see Lemma \ref{Lemma:OrientSets}) through the arc $\gamma_j$, and then traverses
 $G_j$ and $G_k$ in this order (for it cannot cross $G_j\cap G_k$). The claim follows because $\overline{a_hx_{j,k}}$ cannot enter $G_j$ from within $\Delta(a_ja_k,F_j,F_k)$ (i.e., through $\gamma_j$), so its first intersection with $G_k$ must lie outside $\triangle a_kx_{j,k}a_j$ (i.e., after the intersection $a_hx_{j,h}\cap a_jx_{j,k}$).
\end{proof}

For each $1\leq j\leq n$, we fix a point $b_j\in G_{n+1}\cap G_j$ and use $\ell_j$ to denote the line through $a_jb_j$. Applying Claim \ref{Claim:CrissCross} to all triples $G_j,G_h,G_{n+1}$, for $1\leq j\leq h\leq n$, yields that any two of the segments $a_jb_j\subset K_j\cap \ell_j$ and $a_hb_h\subset K_h\cap \ell_h$ intersect. Hence, the sequence $G_1,\ldots,G_m$ is realized by the line sequence $\ell_1,\ldots,\ell_n$. Lastly, applying Claim \ref{Claim:CrissCross} to each triple $G_j,G_h,G_{k}$, for $1\leq j\leq h\leq k\leq n$, and choosing $x_{j,k}=\ell_j\cap \ell_k$ and $x_{h,k}=\ell_{h}\cap \ell_k$, yields that the respective lines $\ell_j,\ell_h,\ell_k$ indeed form a 3-cap. $\Box$

\section{Proof of Theorem \ref{Theorem:SeparateMany}}\label{Sec:Separate}

At the heart of our argument lies the following property.

\begin{lemma}\label{Lemma:Separate}
There is a positive integer $J\geq 5$ with the following property.
Let $\ell_1,\ell_2,\ldots,\ell_J$ be a monotone sequence of $J$ lines in $\reals^3$, so that $\ell_1\succ \ell_2\succ \ldots\succ \ell_J$.	Then there exist 4 distinct indices $i,j,k,l\in [J]$ so that $\ell_{i}$ and $\ell_j$ are separated by the plane $\pi(\ell_k,\ell_l)$.
\end{lemma}
\noindent{\bf Proof of Lemma \ref{Lemma:Separate}.} Let $J=R_3(5,4)$, where $R_k(m,c)$ denotes the Ramsey number in Theorem \ref{Theorem:Ramsey}, and fix a monotone sequence $\ell_1\succ\ell_2\succ \ldots\succ \ell_J$ of $J$ lines. Up to rotating the $xy$ coordinate frame, it can be assumed with no loss of generality that none of the lines $\ell_i$ is parallel to the $x$-axis, so they can be oriented in the positive direction of the $x$-axis.
Assume with no loss of generality that the projected sequence $\ell^*_1,\ldots,\ell^*_J$ forms a $J$-cap. For any $1\leq i<j\leq J$ let $V_{i,j}$ denote the vertical prism over the wedge $(\ell^*_i)^-\cap (\ell^*_j)^-$, and $Z_{i,j}$ denote the open region of $V_{i,j}$ that lies inbetween $\pi(\ell_i,\ell_j)$ and $\pi(\ell_j,\ell_i)$ (see Figure \ref{Fig:lines-3d}). Let $\Delta_{i,j}^+$ (resp., $\Delta_{i,j}^-$) denote the top (resp., bottom) facet of $Z_{i,j}$, which is contained in the plane $\pi(\ell_i,\ell_j)$ (resp., $\pi(\ell_j,\ell_i)$).
	

	\begin{figure}[htb]
 \begin{center}
 	  	 	\includegraphics[scale=0.40]{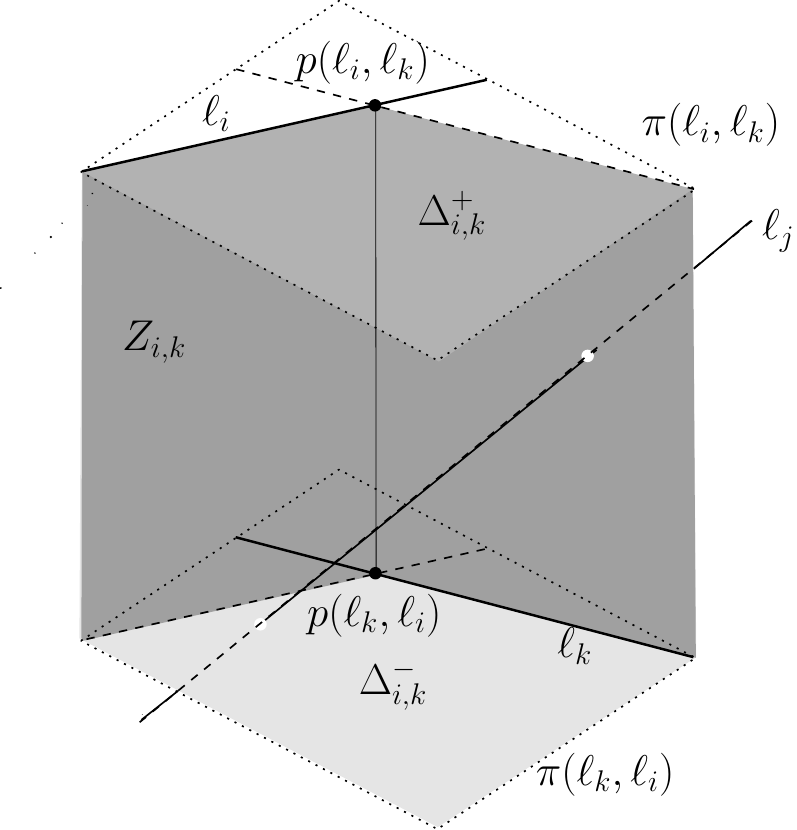}\hspace{2cm}\includegraphics[scale=0.40]{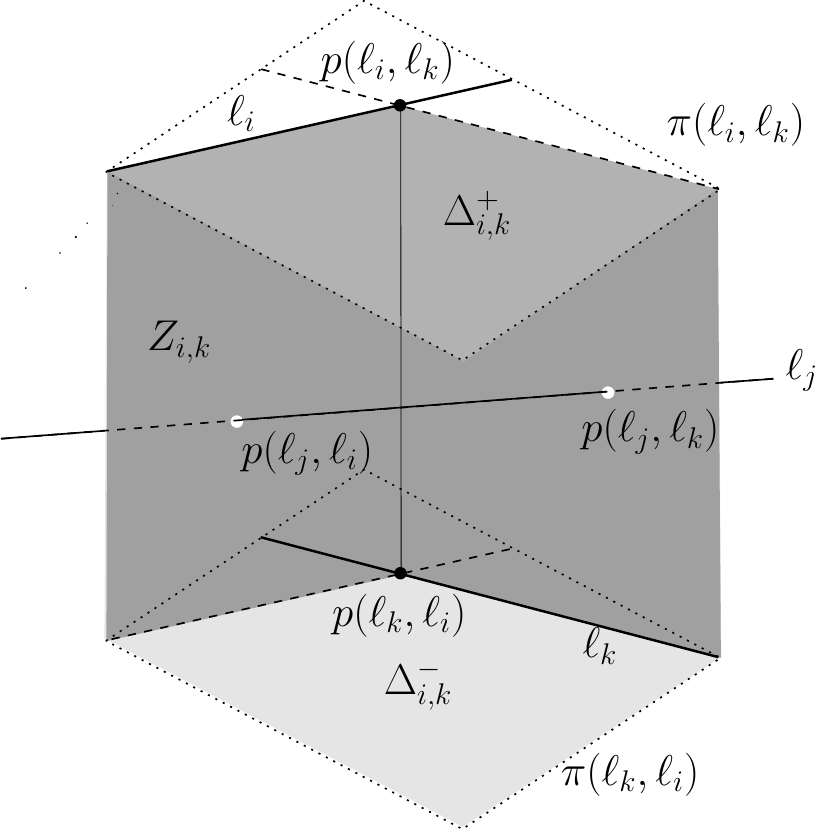} 	 
 	\caption{\small Proposition \ref{Prop:3Lines}. Depicted is the prism $Z_{i,k}$ which lies above $(\ell^*_i)^-\cap (\ell^*_k)^-$, and is ``sandwiched" between the planes $\pi(\ell_i,\ell_k)$ and $\pi(\ell_k,\ell_i)$. Left: In case 1, the line $\ell_j$ enters $Z_{i,k}$ through its ``bottom facet" $\Delta_{i,k}^-$. Right: In cases 3 and 4, the line $\ell_j$ crosses neither of the facets $\Delta_{i,k}^-,\Delta_{i,k}^+$, so that $\ell_j\cap Z_{i,k}=p(\ell_j,\ell_i)p(\ell_j,\ell_k)$ holds.}\label{Fig:lines-3d}
 \end{center}
 \end{figure}

\begin{proposition}\label{Prop:3Lines}
 With the previous assumptions, any triple of lines $(\ell_i,\ell_j,\ell_k)$, with $1\leq i<j<k\leq J$, falls into at least one of the following categories:

\begin{enumerate}
 \item $\ell_j$ enters $Z_{i,k}$ at a point $p(\ell_j,\ell_i)\in \Delta_{i,k}^-$. That is, the point $p(\ell_j,\ell_i)$ lies below $\pi(\ell_k,\ell_i)$.
 \item $\ell_j$ leaves $Z_{i,k}$ at a point $p(\ell_j,\ell_k)\in \Delta_{i,k}^+$. That is, the point $p(\ell_j,\ell_k)$ lies above $\pi(\ell_i,\ell_k)$.
\item $\ell_j$ crosses neither of the facets $\Delta_{i,k}^-$ or $\Delta_{i,k}^+$, so that $\ell_j\cap Z_{i,k}=p(\ell_j,\ell_i)p(\ell_j,\ell_k)$. Furthermore, $\ell_j$ traverses $\pi(\ell_i,\ell_k)$, $Z_{i,k}$ and $\pi(\ell_k,\ell_i)$ in this order.

 \item $\ell_j$ crosses neither of the facets $\Delta_{i,k}^-$ or $\Delta_{i,k}^+$. Furthermore, $\ell_j$ traverses $\pi(\ell_k,\ell_i), Z_{i,k}$ and $\pi(\ell_i,\ell_k)$ in this order.
 \end{enumerate}
 \end{proposition}
 \begin{proof}[Proof of Proposition \ref{Prop:3Lines}.]
Clearly, the line $\ell_{j}$ must intersect $Z_{i,k}$ or, else, it would pass either entirely above or entirely below $Z_{i,k}$. Since both points $\ell_i^*\cap \ell_j^*$ and $\ell_j^*\cap \ell^*_k$ lie on the boundary of $Z^*_{i,k}=(\ell_i^*)^-\cap (\ell_k^*)^-$, the former scenario is contrary to the assumption that $\ell_i\succ\ell_j$, and the latter is contrary to the assumption that $\ell_j\succ \ell_k$. 

If $\ell_j$ crosses any of the facets $\Delta^+_{i,k}$ or $\Delta^-_{i,k}$, then we encounter one of the scenarios 1 or 2. If this is not the case, then we encounter one of the scenarios 3 or 4, depending on whether the directed line $\ell_j$ crosses the plane $\pi(\ell_i,\ell_k)$ from $\pi^+(\ell_i,\ell_k)$ to $\pi^-(\ell_i,\ell_k)$ or vice versa.
 \end{proof}

	By the choice of $J$ in accordance with Theorem \ref{Theorem:Ramsey}, there must exist 5 lines $\ell_{i_1}\succ \ell_{i_2}\succ \ell_{i_3}\succ \ell_{i_4}\succ \ell_{i_5}$ so that every three of them fall into the same scenario of Proposition \ref{Prop:3Lines}. Assume with no loss of generality that $i_1=1,\ldots,i_5=5$ (see Figure \ref{Fig:5lines}).
	
	\begin{figure}[htb]
 \begin{center}
 	 	\includegraphics[scale=0.5]{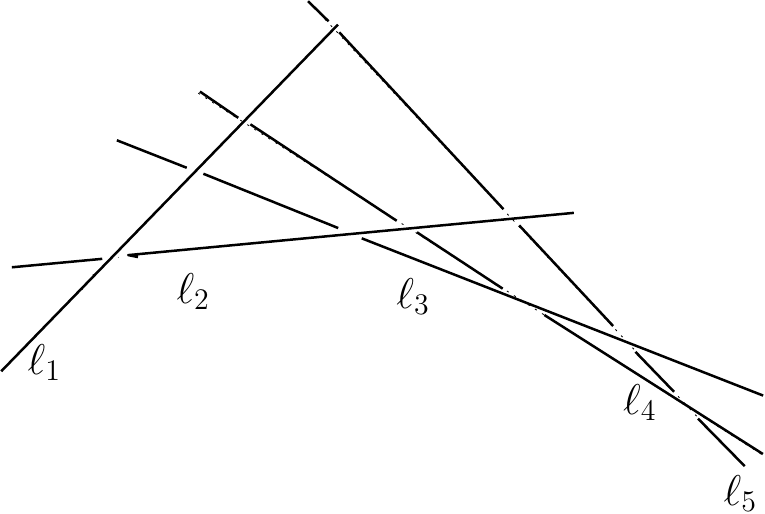} 	 
 	\caption{\small The 5 lines $\ell_{1}\succ \ell_{2}\succ \ell_{3}\succ \ell_{4}\succ \ell_{5}$ -- view from above.}\label{Fig:5lines}
 \end{center}
 \end{figure}

	\medskip
\noindent{\bf Case 1.}
All of the triples $(\ell_{i_1},\ell_{i_2},\ell_{i_3})$, for $1\leq i_1<i_2<i_3\leq 5$, are of type 1. 
In particular, the line $\ell_3$ enters $Z_{2,4}$ through the bottom facet $\Delta^-_{2,4}$.
Notice that (i) the point $\ell_1^*\cap \ell_3^*$ lies in the wedge $W_{2,4}:=(\ell^*_2)^+\cap (\ell^*_4)^-$, (ii) $\ell_1$ passes above the plane $\pi(\ell_4,\ell_2)$ as it traverses the prism $W^\star_{2,4}=\{x\in \reals^3\mid x^*\in W\}$ (for $\ell_1$ passes above both lines $\ell_2$ and $\ell_4$) , (iii) the interval $\ell_3\cap W^\star_{2,4}$ lies below $\pi(\ell_4,\ell_2)$ (because the triple $\ell_2,\ell_3,\ell_4$ is of type 1, so that $\ell_3$ traverses $Z_{2,4}$ after $W^\star_{2,4}$, entering $Z_{2,4}$ from below through its bottom facet $\Delta_{2,4}^-\subset \pi(\ell_4,\ell_2)$).  Hence, the line $\ell_3$ is separated from $\ell_1$ by the plane $\pi(\ell_4,\ell_2)$.

	\medskip
\noindent{\bf Case 2.}
If all of the triples $(\ell_{i_1},\ell_{i_2},\ell_{i_3})$, for $1\leq i_1<i_2<i_3\leq 5$, are of type 2 then, in particular, the line $\ell_3$ leaves $Z_{2,4}$ through its top facet $\Delta_{2,4}^+$. Hence, arguing in a fully symmetrical manner to case 1 yields that $\ell_3$ is separated by $\pi(\ell_2,\ell_4)$ from $\ell_5$. 


	\medskip
\noindent{\bf Case 3.} All of the triples $(\ell_{i_1},\ell_{i_2},\ell_{i_3})$, for $1\leq i_1<i_2<i_3\leq 5$, are of type 3. For any $1\leq i\neq j\leq 5$, let us use $\ell_{i,j}$ to denote the vertical line over the point $\ell^*_i\cap \ell_j^*$. Note that the three planes $\pi(\ell_3,\ell_2),\pi(\ell_3,\ell_4)$, and $\pi(\ell_3,\ell_5)$ through  $\ell_3$ are traversed by $\ell_{2,4}$ and $\ell_{2,5}$ in the same order, as both $\ell^*_{2,4}$ and $\ell^*_{2,5}$ lie above $\ell_3^*$.
     The crucial observation is that each of the lines $\ell\in\{\ell_{2,4},\ell_{2,5}\}$ crosses $\pi(\ell_3,\ell_2)$ (and, therefore, also $\pi(\ell_2,\ell_3)$) {\it above} its intersections $\ell\cap \pi(\ell_3,\ell_4)$ and $\ell \cap \pi(\ell_3,\ell_5)$ with the remaining planes $\pi(\ell_3,\ell_4)$ and $\pi(\ell_3,\ell_5)$.
   
  To see this, let us first show that both $\ell_{2,4}$ and $\ell_{2,5}$ traverse the planes $\pi(\ell_3,\ell_2)$ and $\pi(\ell_3,\ell_4)$ in this downward order. Let
   $a=p(\ell_3,\ell_2)$ and $b=p(\ell_3,\ell_4)$ denote the points at which $\ell_3$, respectively, enters and leaves $Z_{2,4}$ (in accordance with Proposition \ref{Prop:3Lines}), and note that $a^*\in \ell_2^*$ and $b^*\in \ell_4^*$.
   Consider the downward translation of a copy $\ell'_2$ of $\ell_2$ until it hits $\ell_3$ at the point $a$, and a symmetric upward translation of a line $\ell'_4$, which at first coincides with $\ell_4$, until it hits $\ell_3$ at the point $b$. 
Denote $\pi'_{24}:=\pi(\ell'_2,\ell_4)$ and $\pi'_{42}:=\pi(\ell'_4,\ell_2)$.
   
   Since $(\ell_2,\ell_3,\ell_4)$ is of type 3, and $\ell_3$ intersects the planes $\pi(\ell_2,\ell_4)$ and $\pi(\ell_4,\ell_2)$ in this order, the projected intersections $(\pi'_{2,4}\cap \ell_3)^*$ and $(\pi'_{4,2}\cap \ell_3)^*$ are moving within the respective wedges $(\ell_2^*)^+\cap (\ell_4^*)^-$ and $(\ell_2^*)^-\cap (\ell_4^*)^+$.
   Thus, the planes $\pi'_{24}$ and $\pi'_{42}$ are moving towards one another, yet never coincide, within the slab delimited by $\pi(\ell_2,\ell_4)$ and $\pi(\ell_4,\ell_2)$. As a result, in the end of the translation we still have that $\ell'_2\succ \ell'_4$. 
Furthermore, using that (i) $\ell'_2\subset \pi(\ell_3,\ell_2)$ and $\ell'_4\subset \pi(\ell_3,\ell_4)$, and (ii) the line $\ell_2$ (resp., $\ell_4$) lies above (resp., below) $\ell'_2$ (resp., $\ell'_4$), we conclude that $\ell_{2,4}$ must cross $\ell_2$, $\ell'_2\subset \pi(\ell_3,\ell_2)$, $\ell'_4\subset \pi(\ell_3,\ell_4)$, and $\ell_4$ in this downward order. Lastly, since $\ell^*_{2,4}$ and $\ell^*_{2,5}$ lie to the same side (i.e., above) $\ell^*_3$, the line $\ell_{2,5}$ must cross  $\ell_2$, $\pi(\ell_3,\ell_2)$, and $\pi(\ell_3,\ell_4)$ in this downward order.

   
    
  Repeating the previous argument for the triple $(\ell_2,\ell_3,\ell_5)$ (of the same type 3) shows that each of the vertical lines $\ell_{2,5}$ and $\ell_{2,4}$ must cross $\ell_2$, $\pi(\ell_3,\ell_2)$, and $\pi(\ell_3,\ell_5)$ in this downward order, and the point $\ell_{2,5}\cap \ell_5$ lies below $\ell_{2,5}\cap \pi(\ell_3,\ell_5)$. Furthermore, recall that the lines $\ell_{2,4}$ and $\ell_{2,5}$ must traverse $\pi(\ell_3,\ell_4)$ and $\pi(\ell_3,\ell_5)$ in the same (so far unknown) order.
 
 If the point $\ell_{2,4}\cap \pi(\ell_3,\ell_5)$ lies above $\ell_{2,4}\cap \pi(\ell_3,\ell_4)$, then $\ell_{2,4}$ must intersect $\ell_2,\pi(\ell_3,\ell_5),\pi(\ell_3,\ell_4)$, and $\ell_4$ in this downward order; thus, $\ell_2$ and $\ell_4$ are separated by the plane $\pi(\ell_3,\ell_5)$. Otherwise, $\ell_{2,5}\cap \pi(\ell_3,\ell_4)$ must lie above $\ell_{2,5}\cap \pi(\ell_3,\ell_5)$, so that $\ell_{2,5}$ crosses $\ell_2,\pi(\ell_{3},\ell_4),\pi(\ell_3,\ell_5)$, and $\ell_5$ in this downward order; hence, the lines $\ell_2$ and $\ell_5$ are separated by the plane $\pi(\ell_3,\ell_4)$.
  
 \medskip
\noindent{\bf Case 4.}  All of the triples $(\ell_{i_1},\ell_{i_2},\ell_{i_3})$, for $1\leq i_1<i_2<i_3\leq 5$, are of type 4, then $\ell_2$ is separated from $\ell_4$ by $\pi(\ell_3,\ell_5)$. 

To see this, let us first consider the triple $(\ell_3,\ell_4,\ell_5)$. Since $\ell_4$ crosses $\pi(\ell_5,\ell_3)$ before $\pi(\ell_3,\ell_5)$, it follows that the plane $\pi(\ell_3,\ell_5)$ lies below (resp., above) $\pi(\ell_3,\ell_4)$ over $(\ell_3^*)^-$ (resp., $(\ell_3^*)^+$). Since $\ell_4$ lies entirely below $\pi(\ell_3,\ell_4)$, it follows that the point $p(\ell_4,\ell_2)=\ell_4\cap \ell_{2,4}$ (whose projection $\ell^*_{2,4}$ lies above $\ell_3^*$) must lie below the plane $\pi(\ell_3,\ell_5)$. 

It suffices to show that both points $p(\ell_2,\ell_3)=\ell_2\cap \ell_{2,3}$ and $p(\ell_2,\ell_5)=\ell_2\cap \ell_{2,5}$ lie above the plane $\pi(\ell_3,\ell_5)$.
Indeed, since $p(\ell_2,\ell_4)\in p(\ell_2,\ell_3)p(\ell_2,\ell_5)$, it would follow that the point $p(\ell_2,\ell_4)=\ell_{2,4}\cap \ell_2$ too lies above $\pi(\ell_3,\ell_5)$, so the plane $\pi(\ell_3,\ell_5)$ indeed crosses the interval $s(\ell_2,\ell_4)$ between $p(\ell_2,\ell_4)=\ell_{2,4}\cap \ell_2$ and $p(\ell_4,\ell_2)=\ell_{2,4}\cap \ell_4$.

For $p(\ell_2,\ell_3)$, the claim follows immediately as the line $\ell_{2,3}$ meets $\ell_3\subset \pi(\ell_{3},\ell_5)$ below the point $p(\ell_2,\ell_3)=\ell_2\cap \ell_{2,3}$.

To see that $p(\ell_2,\ell_5)$ lies below $\pi(\ell_3,\ell_5)$, recall that the segment $\ell_3\cap Z_{2,5}=p(\ell_3,\ell_2)p(\ell_3,\ell_5)$ lies below $\pi(\ell_2,\ell_5)$ (since $(\ell_2,\ell_3,\ell_5)$ is of type 4). Hence, an upward translation of $\ell_5$ until it hits $\ell_3$ at the point $p(\ell_3,\ell_5)$, yields a line $\ell'_5\subset \pi(\ell_3,\ell_5)$ which meets $\ell_{2,5}$ below the point $p(\ell_2,\ell_5)=\ell_2\cap \ell_{2,5}$ (as both $\ell'_5$ and $p(\ell_3,\ell_2)p(\ell_3,\ell_5)\subset \ell_3$ are confined to the slab between $\pi(\ell_2,\ell_5)$ and $\pi(\ell_5,\ell_2)$). $\Box$

\bigskip
\noindent{\bf Proof of Theorem \ref{Theorem:SeparateMany}.}
It can be assumed with no loss of generality that $n\geq J$ (where $J$ denotes the constant in Lemma \ref{Lemma:Separate}), that $\ell_1\succeq \ell_2\succeq \ldots\succeq \ell_n$, and that $\ell_1,\ldots,\ell_n$ is an $n$-cap.
Furthermore, let us assume with no loss of generality that no pair of lines $\ell_i$ and $\ell_j$ intersect.
\footnote{Otherwise the coplanarities can be resolved by an infinitesimally small perturbation that yields $\ell_1\succ \ell_2\succ \ldots\succ \ell_n$, and following a standard limiting argument described, e.g., in \cite[Theorem 2.2]{Rubin}.}

	We apply  Lemma \ref{Lemma:Separate} to each subsequence $\sigma=(\ell_{h_1},\ell_{h_2},\ldots,\ell_{h_J})$ of $J$ lines, with $1\leq h_1<h_2<\ldots<h_J\leq n$, and label $\sigma$ with the ordered quadruple $(i,j,k,l)$, for distinct $i,j,k,l\in [J]$, if the lines $\ell_{h_i}$ and $\ell_{h_j}$ are separated by $\pi(\ell_{h_k},\ell_{h_l})$.
	By the pigeonhole principle, there must exist such a label $(i,j,k,l)$ that is shared by at least
	
	$$
	\frac{{n\choose J}}{4!{J\choose 4}}=\Omega\left(n^J\right)
	$$
	\noindent $J$-sequences $\sigma=(\ell_{h_1},\ell_{h_2},\ldots,\ell_{h_J})$.
	Assume with no loss of generality that $(i,j,k,l)=(1,2,3,4)$. 
	An additional application of the pigeohole principle yields a $(J-2)$-tuple $\sigma'=(\ell_{h_3},\ell_{h_4},\ldots,\ell_{h_J})$ that is suffix to $\Omega(n^2)$ different $J$-sequences $\sigma=(\ell_{h_1},\ell_{h_2},\ell_{h_3},\ell_{h_4},\ldots,\ell_{h_J})$ which share the same label, giving rise to $\Omega(n^2)$ distinct pairs $(\ell_{h_1},\ell_{h_2})$ that are separated by the same plane $\pi(\ell_{h_3},\ell_{h_4})$. $\Box$

\section{Discussion}\label{Sec:Conclude}

\medskip
\noindent{\bf Separating arbitrary families of lines in $\reals^3$.} Though this is immaterial for the argument in Section \ref{Sec:Main}, one can get rid of the assumption in Theorem \ref{Theorem:SeparateMany} that $\ell_1,\ldots,\ell_n$ form a monotone sequence.

\begin{theorem}\label{Theorem:SeparateGeneral}
	Let $\L=\{\ell_1,\ldots,\ell_n\}$ be a family of $n\geq 3$ lines in $\reals^3$. Then there exists a plane separating $\Theta(n^2)$ of the pairs $\ell_i,\ell_j$, for $1\leq i<j\leq n$.
\end{theorem}
\begin{proof}
As before, it can be assumed with no loss of generality that no two lines in $\L$ intersect, no three in the projected family $\L^*=\{\ell_1^*,\ldots,\ell_n^*\}$ pass through the same point, no two of them are parallel, and none of the lines $\ell_i^*$ is parallel to the $y$-axis.

Let $t=R_3(R_2(J,2),2)$, where $R_k(m,c)$ denotes the Ramsey number in Theorem \ref{Theorem:Ramsey}, and $J$ denotes the constant in Lemma \ref{Lemma:Separate}.
The crucial observation is that any $t$-size subset $\M=\{\ell_{i_1},\ldots,\ell_{i_t}\}\in {\L\choose t}$, for $i_1<\ldots<i_t$, must contain a monotone sub-sequence of size $J$. Indeed, the projected sequence $\ell_{i_1},\ldots,\ell_{i_t}$ must contain either a $t$-cap or a $t$-cup $\ell^*_{j_1},\ldots,\ell^*_{j_t}$ of size $t=R_2(J,2)$
 and, moreover, this chain must contain a monotone sub-sequence of size $J$. (The estimate can be improved \cite[Section 3]{JirkaBook}.) 
 
Plugging this into Lemma \ref{Lemma:Separate} implies that every subset $\{\ell_{i_1},\ldots,\ell_{i_t}\}\in {\L\choose t}$ must encompass an ordered quadruple $(\ell_{i_a},\ell_{i_b},\ell_{i_c},\ell_{i_d})$ so that $\ell_{i_a}$ and $\ell_{i_b}$ are separated by $\pi(\ell_{i_c},\ell_{i_d})$. Now the theorem follows by repeating the pigeonhole argument in the proof of Theorem \ref{Theorem:SeparateMany}, albeit with $t$ instead of $J$.
 \end{proof}

\noindent{\bf Open problems.} 
Let us conclude the paper with a list of closely related, and so far open, problems.

\begin{enumerate}
	\item It remains a major open problem to establish the ``2-colored" Conjecture \ref{Conj:2colored} of Mart\'inez, Rold\'an-Pensado, and Rubin \cite{FurtherConsequences}.
	
	The primary obstacle to proving Conjecture \ref{Conj:2colored} via our machinery is that Theorem \ref{Theorem:SeparateGeneral} admits no bi-partite extension. For example, consider two line families $\L_1=\{\ell_i:=\{(i,t,it)\mid t\in \reals\}\}_{i=1}^n$ and $\L_2=\{\ell'_i:=\{(t,i,it)\mid t\in \reals\}\}_{i=1}^n$, each drawn from a different ruling of the same hyperbolic paraboloid $z=yx$. Since (i) any two lines $\ell_i\in \L_1,\ell'_j\in \L_2$ intersect, yet (ii) no four of these intersection points $\ell_i\cap \ell'_j$ are co-planar, a sufficiently small perturbation of $\L_1\cup \L_2$ yields a generic family $\L_1\cup \L_2$ in which no three of the bi-colored pairs $(\ell_i,\ell'_j)\in \L_1\times \L_2$ can be simultaneously separated by the same plane. 
	
	\item It seems to be a rather more tangible goal to show that, under the hypothesis of Theorem \ref{Theorem:Main}, the entire family $\K$ can be crossed by $\Theta(1)$ lines. 
	
	\begin{conjecture}\label{Conj:FixedSize}
		Let $\K$ be a finite-size family of pairwise intersecting convex sets in $\reals^3$. Then $\K$ admits a transversal by $O(1)$ lines.
	\end{conjecture}
	
	The known general methods \cite{AKMM} for using ``fractional" transversal results (in the mold of Theorem \ref{Theorem:Main}) to yield fixed-size transversals rely on the existence so called {\it weak $\eps$-nets}. 
	Specializing to the problem at hand, this calls for the following hypothetic statement: {\it For any $\eps>0$ there is a finite number $N(\eps)$ with the following property: for any finite family $\L$ of lines in $\reals^3$, the family of all convex sets that are crossed by $\eps|\L|$ elements of $\L$ admits a transversal by only $N(\eps)$ lines (not necessarily from $\L$).}
	
	Since the general existence of such ``line $\eps$-nets" has been disproved by Cheong, Goaoc, and Holmsen \cite{CheongGoaocHolm}, the prospective proof of Conjecture \ref{Conj:FixedSize} calls for a more refined (and probably more ad-hoc) approach.
	
	\item As was pointed out by B\'ar\'any and Kalai \cite{BaranyKalaiSurvey}, the just resolved Conjecture \ref{Conj:Pairwise} comprises {\it ``the first,
and so far most interesting, unsolved case of a series of problems of the same type.
Namely, for what numbers $k,r,d$ is it true that, given a family ${\cal C}$ of convex sets
in $\reals^d$, where every $k$-tuple is intersecting, there is an $r$-flat intersecting a positive
fraction of the sets in ${\cal C}$? Of course, the positive fraction should depend only on $k,r$ and $d$."}

	In this vein, it would be instrumental to gain a better understanding of arrangements of {\it strictly $d$-intersecting familes} of convex sets in dimension $d\geq 3$ (that is, families in which any $d$ sets have a non-empty common intersection, yet no $d+1$ have).
\end{enumerate}

\section*{Acknowledgement}

The author thanks SODA referees for many thoughtful comments which helped to improve the presentation.

\appendix
\section{Proof of Lemma \ref{Lemma:DefinedTangent}}\label{App:FracDefined}
We use induction in the dimension $d$, and distinguish between two scenarios, where the second scenario can arise only for $d\geq 2$. Let $\Lambda(\K,d+1)$ denote the collection of all the $(d+1)$-size families $\K'\in {\K\choose d+1}$ that can be (each) crossed by a hyperplane.

If at least $\alpha{n\choose d+1}/2$ of the member families $\K'\in \Lambda(\K,d+1)\subseteq {\K\choose d+1}$ are good, then we invoke Lemma \ref{Lemma:Tangents} for each subset $\K'$ in order to assign each good family $\K'\in \Lambda(\K,d+1)$ a defining subset $\B_{\K'}$ of only $d$ elements so that at least one of their $2^d$ common tangents crosses the remaining element of $\K'$. Furthermore, the pigeonhole principle yields such a set $\B\in {\K\choose d}$ that is contained in ${n\choose d+1}/\left(2{n\choose d}\right)$ families $\K'\in \Lambda(\K,d+1)$ with $\B_{\K'}=\B$, and a hyperplane $\pi\in \HH(\B)$ crossing ${n\choose d+1}/\left(2\cdot 3^{d}\cdot {n\choose d}\right)=\Omega(n)$ elements of $\K\setminus \B$.

If at least $\alpha{n\choose d+1}/2$ of the $(d+1)$-families $\K'\in \Lambda(\K,d+1)$ are bad, applying Lemma \ref{Lemma:Tangents} to each of them gives rise to a total of at least $\alpha{n\choose d+1}/(2n)=\Omega(\alpha n^d)$ subsets $\G\in {n\choose d}$, each crossed by a $(d-2)$-flat. We project the sets of $\K$ to the $x_{d}$-orthogonal copy of $\reals^{d-1}$ and observe that the resulting family $\K^*$ encompasses at least $\alpha' n^d$ $d$-size subsets, each crossed by a single $(d-2)$-plane within $\reals^{d-1}$; here $0<\alpha'\leq 1$ is a suitable constant that depends only on $\alpha$ and $d$. Hence, the induction hypothesis yields a $(d-2)$-plane $\mu$ within $\reals^{d-1}$ that is determined by a certain subset $\B^*\in {\K^* \choose d-1}$ (with $\B\in {\K\choose d-1}$), and crossing $\beta'_{d-1}(\alpha')n$ elements of $\K^*$. Hence, the ``lifted" hyperplane $\pi=\varphi_{d,d-1}^{-1}(\mu)$ in $\reals^d$ crosses $\Omega(|\K|)$ elements of the original family $\K$, and it is determined by {\it any} $d$-size set $\C\in {\K\choose d}$ that contains $\B$. $\Box$

\end{document}